\documentclass[final]{siamltex}

\usepackage[usenames]{color}
\usepackage[color]{showkeys}
\definecolor{refkey}{gray}{0.4}
\definecolor{labelkey}{gray}{0.3}

\usepackage[left=1in,top=1in,right=1in,bottom=1in,dvips,letterpaper]{geometry}
\usepackage{setspace,url,changepage,multirow,rotating}
\usepackage[numbers,square,sort]{natbib}
\usepackage[leqno]{amsmath}
\usepackage{amsxtra, amsfonts,amscd,  amssymb, graphicx,subfigure,enumerate}

\usepackage[algo2e, vlined,ruled]{algorithm2e}
\usepackage{algorithmic,algorithm}
\numberwithin{equation}{section}

\def\grad{\nabla}
\def\cO{\mathcal{O}}
\newcommand{\integers}{\mathbb{Z}}
\newcommand{\proc}[1]{\textnormal{\scshape#1}}
\newcommand{\argmin}{\mathop{\rm argmin}}

\newcommand{\be}{\begin{equation}}
\newcommand{\ee}{\end{equation}}
\newcommand{\ba}{\begin{array}}
\newcommand{\ea}{\end{array}}
\newcommand{\bea}{\begin{eqnarray}}
\newcommand{\eea}{\end{eqnarray}}
\newcommand{\beaa}{\begin{eqnarray*}}
\newcommand{\eeaa}{\end{eqnarray*}}

\newcommand{\half}{\frac{1}{2}}
\newcommand{\br}{\mathbb{R}}

\newcommand{\mtn}{{m \times n}}

\usepackage[usenames]{color}
\usepackage[normalem]{ulem}

\newcommand{\LCal}{\mathcal{L}}

\newcommand{\sgn}{\mathrm{sgn}}

\newcommand{\etal}{{et al. }}
\newcommand{\Diag}{\mbox{Diag}}

\newcommand{\vvec}{\mbox{vec}}
\newcommand{\reals}{\br}
\newcommand{\norm}[1]{\| #1 \|}

\begin{document}

\title{Efficient Algorithms for Robust and Stable Principal Component Pursuit Problems}
\date{}
\author{Necdet Serhat Aybat\footnotemark[1] \and Donald Goldfarb\footnotemark[2] \and Shiqian Ma\footnotemark[3]}
\renewcommand{\thefootnote}{\fnsymbol{footnote}}

\footnotetext[1]{Department of Industrial Engineering, Penn State University. Email: nsa10@psu.edu}
\footnotetext[2]{Department of Industrial Engineering and Operations
Research, Columbia University. Email: goldfarb@columbia.edu.} 
\footnotetext[3]{Department of Systems Engineering and Engineering Management, The Chinese University of Hong Kong, Shatin, Hong Kong. Email: sqma@se.cuhk.edu.hk}

\renewcommand{\thefootnote}{\arabic{footnote}}

\maketitle 

\begin{abstract} The problem of recovering a low-rank matrix from a set of observations corrupted with gross sparse error is known as the robust principal component analysis (RPCA) and has many applications in computer vision, image processing and web data ranking. It has been shown that under certain conditions, the solution to the NP-hard RPCA problem can be obtained by solving a convex optimization problem, namely the robust principal component pursuit (RPCP). Moreover, if the observed data matrix has also been corrupted by a dense noise matrix in addition to gross sparse error, then the stable principal component pursuit (SPCP) problem is solved to recover the low-rank matrix. In this paper, we develop efficient algorithms with provable iteration complexity bounds for solving RPCP and SPCP. Numerical results on problems with millions of variables and constraints such as foreground extraction from surveillance video, shadow and specularity removal from face images and video denoising from heavily corrupted data show that our algorithms are competitive to current state-of-the-art solvers for RPCP and SPCP in terms of accuracy and speed. \end{abstract}



\section{Introduction}
Principal component analysis (PCA) plays an important role in applications arising from image and video processing, web data analysis and bioinformatics. PCA obtains a low-dimensional approximation to high-dimensional data in the $\ell_2$ sense, by computing the singular value decomposition (SVD) of a matrix. However, when the given data is corrupted by gross errors, classical PCA becomes impractical because the grossly corrupted observations can jeopardize the $\ell_2$ estimation. To overcome this shortcoming, a new model called robust PCA (RPCA) was considered by Wright \etal \cite{Wright-Ganesh-Rao-Peng-Ma-NIPS-2009}, Cand\`{e}s \etal \cite{Candes-Li-Ma-Wright-RPCA-2009} and Chandrasekaran \etal \cite{Chandrasekaran-Sanghavi-Parrilo-Willsky-2009} under the assumption that the gross errors are sparse. In this model it is assumed that the data matrix $D\in\reals^{m\times n}$ is of the form $D := L^0 + S^0$, where $L^0$ is a low-rank matrix, i.e., $\rank(L^0)\ll\min\{m,n\}$, and $S^0$ is a sparse matrix, i.e., $\norm{S^0}_0\ll mn$, where the so-called $\ell_0$-norm $\|S\|_0:=\|\vvec(S)\|_0$, $\vvec(S)$ is the vector obtained by stacking the columns of $S$ in order and $\norm{.}_0$ counts the number of nonzero elements of its argument.

To obtain a low rank and sparse decomposition of a given matrix $D$, RPCA combines both the $\rank$ function and the $\ell_0$-norm in the objective with some weighting parameter $\xi>0$ to balance the weights of rank and sparsity. This leads to the following formulation of RPCA:
\bea\label{prob:RPCA-rank-L0} \min_{L,S\in\br^\mtn} \{ \rank(L) + \xi\|S\|_0 : L + S = D \}. \eea
The RPCA problem is related to the matrix rank minimization \cite{Fazel-thesis-2002,Recht-Fazel-Parrilo-2007,Candes-Recht-2008,Candes-Tao-2009} problem, which itself is a generalization of the recovery problem in compressed sensing \cite{Donoho-06,Candes-Romberg-Tao-2006}. As in the compressed sensing and matrix rank minimization problems, \eqref{prob:RPCA-rank-L0} is NP-hard due to the combinatorial nature of the $\rank$ function and the $\ell_0$-norm.

Recently, it has been shown that under certain probabilistic conditions on $D:= L^0 + S^0$, 
with very high probability, the low-rank matrix $L^0$ and the gross sparse ``error'' matrix $S^0$ can be recovered by solving the \emph{robust principal component pursuit} (RPCP) problem \citep{Candes-Li-Ma-Wright-RPCA-2009}:
\bea\label{prob:RPCA-Nuclear-L1} \min_{L,S\in\br^\mtn} \{ \|L\|_*+\xi\|S\|_1 : L+S=D \},\eea where $\xi=\frac{1}{\sqrt{\max\{m,n\}}}$, $\|L\|_*$ denotes the nuclear norm of $L$, which is defined as the sum of the singular values of $L$, and $\|S\|_1:=\|\vvec(S)\|_1$. Moreover, problem \eqref{prob:RPCA-Nuclear-L1} is also studied in \citep{Chandrasekaran-Sanghavi-Parrilo-Willsky-2009}, where a deterministic condition for exact recovery for is provided.

In \cite{Zhou-Candes-Ma-StablePCP-2010}, it is shown that the recovery is still possible even when the data matrix, $D:=L^0 + S^0+N^0$, is corrupted with a dense error matrix, $N^0$ such that $\|N^0\|_F\leq\delta$. Indeed, any optimal solution $(L^*,S^*)$ to the \emph{stable principal component pursuit}~(SPCP) problem
\be\label{prob:SPCP} \qquad  \min_{L,S\in\br^{m\times n}}\{\|L\|_*+\xi\|S\|_1:\ \|L+S-D\|_F\leq\delta\}, \ee satisfies $\norm{L^*-L^0}_F^2+\norm{S^*-S^0}_F^2=\cO(\delta^2)$ with very high probability. Note that RPCP problem \eqref{prob:RPCA-Nuclear-L1} is a special case of SPCP problem \eqref{prob:SPCP} with $\delta=0$.

In some applications, some of the entries of $D$ in \eqref{prob:SPCP} may be missing. Let $\Omega\subset\{i:1\leq i\leq m\}\times\{j:1\leq j\leq n\}$ be the index set of the entries of $D$ that are observable and define the projection operator $\pi_{\Omega}:\reals^{m\times n}\rightarrow \reals^{m\times n}$ as $(\pi_\Omega(L))_{ij}=L_{ij}$, if $(i,j)\in\Omega$ and $(\pi_\Omega(L))_{ij}=0$ otherwise. Note that $\pi^*_{\Omega}(.)=\pi_{\Omega}(.)$, where $\pi^*_{\Omega}$ denotes the adjoint operator. In these applications with missing data, the problem
\be\label{prob:SPCP-missing} \qquad \min_{L,S\in\br^{m\times n}}\{\|L\|_*+\xi\|S\|_1:\ \|\pi_\Omega(L+S-D)\|_F\leq\delta\}, \ee
is solved to recover the low rank and sparse components of $D$.
It has been shown under some randomness hypotheses that the low rank $L^0$ and sparse $S^0$ can be recovered with high probability by solving \eqref{prob:SPCP-missing} when $\delta=0$ (see Theorem 1.2 in \citep{Candes-Li-Ma-Wright-RPCA-2009}).
In this paper, we will provide efficient methods to solve a problem equivalent to \eqref{prob:SPCP-missing}. In the following theorem, we state the alternative formulation and establish its equivalence to \eqref{prob:SPCP-missing}.
\begin{theorem}\label{thm:SPCP-missing-equiv}
$(L^*,\pi_\Omega(S^*))$ is an optimal solution to \eqref{prob:SPCP-missing} if $(L^*,S^*)$ is an optimal solution to
\bea\label{prob:SPCP-missing-equiv} \min_{L,S\in\br^{m\times n}}\{ \|L\|_* + \xi\|\pi_\Omega(S)\|_1 : \norm{L + S - \pi_\Omega(D)}_F\leq \delta\}. \eea
\end{theorem}
\begin{proof}
Suppose $(\bar{L},\bar{S})$ is an optimal solution to \eqref{prob:SPCP-missing}. We claim that $\bar{S}_{ij}=0,~\forall~(i,j)\notin\Omega$. Otherwise, $(\bar{L},\pi_\Omega(\bar{S}))$ is feasible to \eqref{prob:SPCP-missing} and has a strictly smaller objective function value than $(\bar{L},\bar{S})$, which contradicts the optimality of $(\bar{L},\bar{S})$. Thus, $\|\pi_\Omega(\bar{S})\|_1=\|\bar{S}\|_1$. Let $(L^*,S^*)$  be an optimal solution to \eqref{prob:SPCP-missing-equiv}. Now suppose that $(L^*,\pi_\Omega(S^*))$ is not optimal to \eqref{prob:SPCP-missing}; then we have \bea\label{proof:SPCP-missing-eq1}\|\bar{L}\|_*+\xi\|\bar{S}\|_1 = \|\bar{L}\|_*+\xi\|\pi_\Omega(\bar{S})\|_1 < \|L^*\|_* +\xi\|\pi_\Omega(S^*)\|_1. \eea By defining a new matrix $\tilde{S}$ as \beaa\tilde{S}_{ij}=\left\{\ba{ll} \bar{S}_{ij}, & (i,j)\in\Omega \\ -\bar{L}_{ij}, & (i,j)\notin\Omega, \ea\right.\eeaa we have that $(\bar{L},\tilde{S})$ is feasible to \eqref{prob:SPCP-missing-equiv} and $\|\pi_\Omega(\tilde{S})\|_1=\|\pi_\Omega(\bar{S})\|_1$. Combining this with \eqref{proof:SPCP-missing-eq1}, we obtain \beaa\|\bar{L}\|_*+\xi\|\pi_\Omega(\tilde{S})\|_1<\|L^*\|_*+\xi\|\pi_\Omega(S^*)\|_1,\eeaa which contradicts the optimality of $(L^*,S^*)$ to \eqref{prob:SPCP-missing-equiv}. Therefore, $(L^*,\pi_\Omega(S^*))$ is optimal to \eqref{prob:SPCP-missing}.
\end{proof}

Although \eqref{prob:SPCP-missing-equiv}
can be reformulated as a semi-definite programming~(SDP) problem and thus, in theory, can be efficiently solved by interior point methods, these solution techniques are impractical when the problem size is large. Recently, algorithms using only first-order information for solving RPCP problem \eqref{prob:RPCA-Nuclear-L1} and SPCP problem \eqref{prob:SPCP-missing} have been proposed. Algorithms for RPCP include the accelerated proximal gradient (APG) method by Lin \etal \citep{Ma09_1R} and alternating direction method of multipliers (ADMM) by Lin \etal \cite{Ma09_1J}, Cand\`{e}s \etal \cite{Candes-Li-Ma-Wright-RPCA-2009}
and Yuan and Yang \citep{Yuan-Yang-2009}. The APG method in \citep{Ma09_1R} is a variant of Nesterov's optimal gradient methods \citep{NesterovConvexBook2004,Nesterov-2005,Beck-Teboulle-2009}. In particular, the APG method  in \citep{Ma09_1R} is essentially the FISTA method in \cite{Beck-Teboulle-2009} applied to RPCP.
Algorithms for SPCP problem include an augmented Lagrangian algorithm FALC~\cite{Ser10_1J} by Aybat and Iyengar; and two alternating direction augmented Lagrangian algorithms: ASALM~\cite{Tao09_1J} by Tao and Yuan and NSA~\cite{Ser11_1J} by Aybat and Iyengar. Indeed, in \cite{Ser11_1J,Tao09_1J}, it is shown that ASALM and NSA iterate sequences converge to an optimal solution of the SPCP problem. However, there is no iteration complexity result for both ASALM and NSA. On the other hand, FALC proposed in ~\cite{Ser10_1J} can be used to solve both RPCP and SPCP problems and computes an $\epsilon$-optimal solution within $\cO(1/\epsilon)$ SVD computations.

{\bf Our contribution.} This paper is dedicated to developing efficient algorithms for solving RPCP and SPCP problems given \eqref{prob:SPCP-missing-equiv} (for RPCP $\delta=0$). We propose several first-order methods and alternating direction type methods for solving \eqref{prob:SPCP-missing-equiv} and analyze their iteration complexity results. We show how our proposed methods can be applied to solve huge problems, involving millions of variables and linear constraints, arising from foreground extraction from surveillance video, shadow and specularity removal from face images and video denoising from heavily corrupted data. We report numerical results on these problems which show that our algorithms are competitive with current state-of-the-art solvers for RPCP and SPCP in terms of accuracy and speed.

{\bf Organization.} The rest of this paper is organized as follows. In section~\ref{sec:smooth}, we briefly describe the smoothing technique we used to smooth at least one of the non-smooth terms in the objective function of the RPCP and SPCP problems. Next, in Section \ref{sec:ADM-RPCA} we briefly review the exact and inexact alternating direction methods for solving RPCP \eqref{prob:RPCA-Nuclear-L1} proposed in \cite{Ma09_1J,Candes-Li-Ma-Wright-RPCA-2009}. In Section \ref{sec:ALM}, we propose our alternating linearization method~(ALM) for solving the RPCP problem \eqref{prob:SPCP-missing-equiv} with $\delta=0$ and present its iteration complexity bound. 
In section~\ref{sec:SPCP}, we show how the generic proximal gradient algorithms can be customized for solving SPCP problem given in \eqref{prob:SPCP-missing-equiv}. Specifically, we show that the subproblems that arise when applying the accelerated proximal gradient method FISTA proposed in 
~\citep{Beck-Teboulle-2009} to the SPCP problem 
have solutions that can be obtained with very modest effort. This result enables us to provide a worst case computational complexity result for the SPCP problem. Numerical results on synthetic and real RPCP and SPCP problems are reported in sections \ref{sec:numerics-RPCP} and \ref{sec:numerics-SPCP}, respectively. 
\section{Smoothing Technique}
\label{sec:smooth}
Note that the objective function of \eqref{prob:SPCP-missing-equiv} is the sum of two non-smooth functions $f(L):=\|L\|_*$ and $g(S):=\xi\|\pi_\Omega(S)\|_1$. The algorithms with provable iteration complexity bounds introduced in this paper require that one or both of the functions $f(L)$ and $g(S)$ are smooth with Lipschitz continuous gradients. Here we adopt Nesterov's smoothing technique \cite{Nesterov-2005} 
to guarantee that the gradient of the smoothed function is Lipschitz continuous.
For fixed parameters $\mu>0$ and $\nu>0$, we define the smooth $C^{1,1}$ functions $f_\mu(.)$ and $g_\nu(.)$ as follows

\begin{align}
&f_\mu(L):=\max_{W\in\reals^{m\times n}:~\norm{W}\leq 1}\left\langle L,W\right\rangle-\frac{\mu}{2}\norm{W}_F^2, \label{eq:smooth_f}\\
&g_\nu(S):=\max_{Z\in\reals^{m\times n}:~\norm{Z}_\infty\leq \xi}\left\langle \pi_\Omega(S),Z\right\rangle-\frac{\nu}{2}\norm{Z}_F^2, \label{eq:smooth_g}
\end{align}
where $\norm{W}$ denotes the spectral norm of $W\in\reals^\mtn$ and $\norm{Z}_\infty:=\norm{\vvec(Z)}_\infty$. We denote the optimal solutions of \eqref{eq:smooth_f} and \eqref{eq:smooth_g} by $W_\mu(L)$ and $Z_\nu(S)$, respectively. Using the rotational invariance of \eqref{eq:smooth_f}, we reduce it to a vector problem: \[\min_{\sigma\in\reals^{\min\{m,n\}}}\{\norm{\sigma-\bar{\sigma}/\mu}_2: \|\sigma\|_\infty\leq 1\},\]
where $\bar{\sigma}$ is a vector whose elements are the singular values of the matrix $L$. This problem has a closed form solution $\sigma^*$ such that $\sigma^*_i=\min\left\{\frac{\bar{\sigma}_i}{\mu},1\right\}$ for $i=1,\ldots,\min\{m,n\}$. Thus \bea\label{def:W-sigma(L)}
W_\mu(L)=U~\Diag\left(\min\left\{\frac{\bar{\sigma}}{\mu},\mathbf{1}\right\}\right)V^\top\eea
gives the solution to \eqref{eq:smooth_f}, where $L=U~\Diag(\bar{\sigma})V^\top$ is the singular value decomposition of $L$ and $\mathbf{1}$ denotes the vector of ones. Similarly, we can show that the solution to \eqref{eq:smooth_g} is \bea\label{def:Z-sigma(Y)}[Z_\nu(S)]_{ij}=\left\{
                                                                   \begin{array}{ll}
                                                                     \sgn(S_{ij})\min\{|S_{ij}|/\nu,\ \xi\} , & \forall~(i,j)\in\Omega; \\
                                                                     0, & \forall~(i,j)\notin\Omega.
                                                                   \end{array}
                                                                 \right.
\eea
According to Theorem 1 in \citep{Nesterov-2005}, the gradient of $f_\mu$ is given by $\nabla f_\mu(L)=W_\mu(L)$ and is Lipschitz continuous with Lipschitz constant 
$1/\mu$; and the gradient of $g_\nu$ is given by $\nabla g_\nu(S)=\pi^*_\Omega(Z_\nu(S))=Z_\nu(S)$ and is Lipschitz continuous with Lipschitz constant $1/\nu$.

In the following sections, we will introduce algorithms with provable iteration complexity bounds that solve problems approximating \eqref{prob:SPCP-missing-equiv} with a smooth objective function
\be\label{prob:SPCP-smooth} \min_{L,S\in\reals^\mtn}\{f_\mu(L)+g_\nu(S):(L,S)\in\chi\}, \ee
and with a partially smooth objective function
\be\label{prob:SPCP-partial-smooth} \min_{L,S\in\reals^\mtn}\{f_\mu(L)+g(S):(L,S)\in\chi\}, \ee
where $\chi:=\{(L,S)\in\reals^{m\times n}\times\reals^{m\times n}:~\norm{L+S-\pi_\Omega(D)}_F\leq\delta\}$.

The inexact solutions of \eqref{prob:SPCP-smooth} and \eqref{prob:SPCP-partial-smooth} are closely related to solution of \eqref{prob:SPCP-missing-equiv}. In fact, let $\tau:=\half\min\{m,n\}$. Then we have $\max\{\half\|W\|_F^2:\|W\|\leq 1\}=\tau$, and $\max\{\half\|Z\|_F^2:\|Z\|_\infty\leq\xi\}=\half mn\xi^2=\tau$, where the second equality follows from $\xi=\frac{1}{\sqrt{\max\{m,n\}}}$. Hence, from \eqref{eq:smooth_f} and \eqref{eq:smooth_g}, we have
\begin{align}
\label{f-sigma-f-bound} f_\mu(L)\leq f(L)\leq f_\mu(L)+\mu \tau, \quad \forall L\in\br^\mtn,\\
\label{g-sigma-g-bound} g_\nu(S)\leq g(S)\leq g_\nu(S)+\nu \tau, \quad \forall S\in\br^\mtn.
\end{align}
\newpage
Therefore, we have the following fact about an $\epsilon$-optimal solution to problem \eqref{prob:SPCP-missing-equiv} ($\hat{x}\in\mathcal{C}$ is called an $\epsilon$-optimal solution of $h^*:=\min_x\{h(x):x\in\mathcal{C}\}$ if $h(\hat{x})-h^*\leq\epsilon$ holds).

\begin{theorem}\label{the:RPCA-Nuclear-L1-epsilon-optimal}
Let $(L^*,S^*)$ be an optimal solution to problem \eqref{prob:SPCP-missing-equiv}. Given $\epsilon>0$, let $(L^*(\mu),S^*(\nu))$ denote an optimal solution to the smoothed problem \eqref{prob:SPCP-smooth} with $\mu=\nu=\frac{\epsilon}{4\tau}$. If $(L(\mu),S(\nu))$ is an $\epsilon/2$-optimal solution to \eqref{prob:SPCP-smooth}, then $(L(\mu),S(\nu))$ is an $\epsilon$-optimal solution to \eqref{prob:SPCP-missing-equiv}.
\end{theorem}
\begin{proof} Using the inequalities in \eqref{f-sigma-f-bound} and \eqref{g-sigma-g-bound}, we have
\begin{eqnarray*}\begin{array}{lll} f(L(\mu))+g(S(\nu))-f(L^*)-g(S^*) & \leq & f_\mu(L(\mu))+g_\nu(S(\nu))+(\mu+\nu)\tau-f_\mu(L^*)-g_\nu(S^*) \\ & \leq & f_\mu(L(\mu))+g_\nu(S(\nu))+(\mu+\nu)\tau-f_\mu(L^*(\mu))-g_\nu(S^*(\nu)) \\ & \leq & \epsilon/2 + (\mu+\nu)\tau = \epsilon/2 + \epsilon/2  =  \epsilon, \end{array}\end{eqnarray*} where the third inequality is due to the fact that $(L(\mu),S(\nu))$ is an $\epsilon/2$-optimal solution to \eqref{prob:SPCP-smooth} and the following equality is due to $\mu=\nu=\frac{\epsilon}{4\tau}$.
\end{proof}

Similarly, with $\mu=\frac{\epsilon}{2\tau}$, an $\epsilon/2$-optimal solution to \eqref{prob:SPCP-partial-smooth} is an $\epsilon$-optimal solution to \eqref{prob:SPCP-missing-equiv}.

In the following sections, we introduce algorithms that find an $\epsilon$-optimal solution to either \eqref{prob:SPCP-smooth} or \eqref{prob:SPCP-partial-smooth} with provable iteration complexity bounds.
\section{Alternating Direction Methods for RPCP}\label{sec:ADM-RPCA}
The alternating direction methods (ADM) in \cite{Ma09_1J,Candes-Li-Ma-Wright-RPCA-2009,Yuan-Yang-2009} are based on an augmented Lagrangian framework. Note that given a penalty parameter $\rho>0$, the augmented Lagrangian function associated with problem \eqref{prob:RPCA-Nuclear-L1} is
\bea\label{RPCA:Lagrangian-function}\LCal_\rho(L,S;\Lambda):=\|L\|_*+\xi\|S\|_1-\langle \Lambda, L+S-D \rangle + \frac{1}{2\rho}\|L+S-D\|_F^2,\eea where $\Lambda$ is a matrix of Lagrange multipliers. Note that the penalty parameter $\rho$ can be adjusted dynamically, and this yields the $k$-th iteration of the augmented Lagrangian method as follows:

\bea\label{RPCA:augmented-Lag-XY}\left\{\ba{rcl}(L_{k+1},S_{k+1})&:=&\argmin_{L,S} \LCal_{\rho_k}(L,S;\Lambda_k)\\ \Lambda_{k+1}&:=&\Lambda_k-(L_{k+1}+S_{k+1}-D)/\rho_k,\\
\rho_{k+1} &:=& \eta\rho_k,\ea\right.\eea
where $\eta\in(0,1]$.

The Exact ADM (EADM) in \cite{Ma09_1J} is based on \eqref{RPCA:augmented-Lag-XY}. However, minimizing \eqref{RPCA:Lagrangian-function} with respect to $L$ and $S$ simultaneously is not easy. In fact, it is as hard as the original problem \eqref{prob:RPCA-Nuclear-L1}. On the other hand, it is easy to minimize $\LCal_\rho(L,S;\Lambda)$ with respect to $L$ or $S$ while keeping the other matrix fixed and each minimization has a closed form solution which is easy to compute. Thus, EADM computes $(L_{k+1},S_{k+1})$ by alternatingly minimizing $\LCal_{\rho_k}(L,S;\Lambda_k)$ repeatedly in $L$ and in $S$, while fixing the other, until the stopping criterion for the inner loop is met, i.e.,
\[\left\{\ba{rcl} L_{k,j+1}&:=&\argmin_L\LCal_{\rho_k}(L,S_{k,j};\Lambda_k), \\
S_{k,j+1}&:=&\argmin_S\LCal_{\rho_k}(L_{k,j+1},S;\Lambda_k), \ea\right.\]
loop is repeated until $\max\{\norm{L_{k,j+1}-L_{k,j}}_F,~\norm{S_{k,j+1}-S_{k,j}}_F\}\leq 10^{-6}~\norm{D}_F$ holds; at that point $(L_{k+1},S_{k+1})$ is set to $(L_{k,j+1}, S_{k,j+1})$. Next, $\Lambda_k$ and $\rho_k$ are updated:
\[\left\{\ba{rcl}\Lambda_{k+1}&:=&\Lambda_k-(L_{k+1}+S_{k+1}-D)/\rho_k,\\
\rho_{k+1} &:=& \eta\rho_k.\ea\right.\]
As a result, the iterate $(L_{k+1},S_{k+1})$ in EADM only approximately minimizes $\LCal_{\rho_k}(L,S;\Lambda_k)$.

Updating the matrix of Lagrangian multipliers $\Lambda$ at every iteration after minimizing $\LCal_{\rho_k}(L,S;\Lambda_k)$ first in $L$ and then in $S$ leads to the following alternating direction method of multipliers~(ADMM). In the $k$-th iteration of ADMM, one computes,
\bea\label{RPCA:ADAL}\left\{\ba{lll}L_{k+1}&:=&\argmin_L\LCal_{\rho_k}(L,S_k;\Lambda_k), \\
                                    S_{k+1}&:=&\argmin_S\LCal_{\rho_k}(L_{k+1},S;\Lambda_k), \\
                                    \Lambda_{k+1}&:=&\Lambda_k-(L_{k+1}+S_{k+1}-D)/\rho_k,\\
                                    \rho_{k+1} &:=& \eta\rho_k,\ea\right.\eea
where $\eta\in(0,1]$. The Inexact ADM (IADM) in \cite{Ma09_1J,Candes-Li-Ma-Wright-RPCA-2009,Yuan-Yang-2009} executes \eqref{RPCA:ADAL} as given.

ADM algorithms can be very efficient since the two minimization subproblems in \eqref{RPCA:ADAL} are easy to solve. Note that the generic form of the subproblem corresponding to $L$ is 
$\min_L\LCal_{\rho}(L,\tilde{S};\tilde{\Lambda})$, 
for some given $\rho$, $\tilde{S}$ and $\tilde{\Lambda}$, which can be reduced to \bea\label{prob:shrinkage-nuclear} \min_L \rho\|L\|_*+\half\|L+\tilde{S}-D-\rho\tilde{\Lambda}\|_F^2.\eea \eqref{prob:shrinkage-nuclear} has a closed-form optimal solution which is given by the matrix shrinkage operator (see, e.g., \cite{Cai-Candes-Shen-2008,Ma-Goldfarb-Chen-2008}) $U~\Diag\left((\sigma-\rho)_+\right)V^\top$, where $U~\Diag(\sigma)V^\top$ is the singular value decomposition of the matrix $(D+\rho\tilde{\Lambda}-\tilde{S})$ and $(.)_+$ is a componentwise operator such that $(a)_+:=\max\{a,0\}$ for all $a\in\reals$.

The generic form of the subproblem corresponding to $S$ is 
$\min_S\LCal_{\rho}(\tilde{L},Y;\tilde{\Lambda})$, 
for some given $\rho$, $\tilde{L}$ and $\tilde{\Lambda}$, which can be reduced to \bea\label{prob:shrinkage-L1} \min_S \rho\xi\|S\|_1+\half\|\tilde{L}+S-D-\rho\tilde{\Lambda}\|_F^2.\eea \eqref{prob:shrinkage-L1} has a closed-form optimal solution, which is given by the vector shrinkage operator (see, e.g., \cite{Daubechies-Defrise-DeMol-04,Hale-Yin-Zhang-SIAM-2008}) $\sgn(D+\rho\tilde{\Lambda}-\tilde{L})\odot\left(|D+\rho\tilde{\Lambda}-\tilde{L}|-\rho\xi\right)_+$, where $\odot$ denotes the componentwise multiplication.

The following convergence result is proved in \cite{Ma09_1J} for EADM, 
i.e., for \eqref{RPCA:augmented-Lag-XY}. On the other hand, no iteration complexity results have been given for EADM when the minimization step in \eqref{RPCA:augmented-Lag-XY} is carried out inexactly, and in practical implementations of EADM, one has to adopt inexact minimizations.

\begin{theorem}\label{the:Exact-ADM-RPCA}{(Theorem 1 in \cite{Ma09_1J})}
Let $\{(L_k,S_k)\}_{k\in\integers_+}$ be the sequence of iterates produced by EADM. Then any accumulation point $(L^*,S^*)$ of $\{(L_k,S_k)\}_{k\in\integers_+}$ is an optimal solution to the RPCP problem~\eqref{prob:RPCA-Nuclear-L1} and the convergence rate is at least $O(\rho_k)$ in the sense that \beaa \left|\|L_k\|_*+\xi\|S_k\|_1- \|L^*\|_* -\xi\|S^*\|_1\right| = O(\rho_{k-1}).\eeaa
\end{theorem}
From this result, it looks like one can obtain any rate result via choosing $\{\rho_k\}_{k\in\integers_+}$ sequence accordingly. However, it is important to note that Theorem~\ref{the:Exact-ADM-RPCA} requires that the optimization problem in \eqref{RPCA:augmented-Lag-XY} be solved exactly; and solving this subproblem for a small value of $\rho_k$ is almost as hard as solving the original RPCP problem~\eqref{prob:RPCA-Nuclear-L1}.

Iteration complexity of IADM, i.e., \eqref{RPCA:ADAL}, is not known. On the other hand, the following convergence result is proved in \cite{Ma09_1J} for IADM.
\begin{theorem}\label{the:Inexact-ADM-RPCA}{(Theorem 2 in \cite{Ma09_1J})}
Let $\{(L_k,S_k)\}_{k\in\integers_+}$ be the sequence of iterates produced by IADM. If $\{\rho_k\}_{k\in\integers_+}$ is nonincreasing and $\sum_{k=1}^{+\infty}\rho_k=+\infty$, then $(L_k,S_k)$ converges to an optimal solution of the RPCP problem~\eqref{prob:RPCA-Nuclear-L1}.
\end{theorem}

In the next section, we propose our alternating linearization method and present an iteration complexity bound for it.
\section{Alternating Linearization Method for RPCP}\label{sec:ALM}
In this section, we introduce the alternating linearization method (ALM) for solving RPCP, i.e., \eqref{prob:SPCP-missing-equiv} with $\delta=0$. For a given penalty parameter $\rho>0$, when $\delta=0$, the augmented Lagrangian function associated with problem \eqref{prob:SPCP-missing-equiv} is
\bea\label{RPCA:Lagrangian-function-missing}\LCal_\rho(L,S;\Lambda):=\|L\|_*+\xi\|\pi_\Omega(S)\|_1-\langle \Lambda, L+S-\pi_\Omega(D) \rangle + \frac{1}{2\rho}\|L+S-\pi_\Omega(D)\|_F^2,\eea where $\Lambda$ is a matrix of Lagrange multipliers.

ALM can be derived from a variant of the ADMM \eqref{RPCA:ADAL}. Note that in each iteration \eqref{RPCA:ADAL} of ADMM, the Lagrange multiplier matrix $\Lambda$ is updated just once, which occurs after the subproblem with respect to $S$ is solved. Since there is no particular reason to minimize the augmented Lagrangian function with respect to $L$ before minimizing it with respect to $S$, it is natural to also update $\Lambda$ after the subproblem with respect to $L$ is solved. By doing this, we get the following symmetric version of the ADMM for a given sequence of penalty multipliers $\{\rho_k\}_{k\in\integers_+}$:
\bea\label{RPCA:ADAL-sym}\left\{\ba{lll}L_{k+1}&:=&\argmin_L\LCal_{\rho_k}(L,S_k;\Lambda_k) \\
                                        \Lambda_{k+\half}&:=&\Lambda_k-(L_{k+1}+S_k-\pi_\Omega(D))/\rho_k \\
                                        S_{k+1}&:=&\argmin_S\LCal_{\rho_k}(L_{k+1},S;\Lambda_{k+\half}) \\
                                        \Lambda_{k+1}&:=&\Lambda_{k+\half}-(L_{k+1}+S_{k+1}-\pi_\Omega(D))/\rho_k.\ea\right.\eea
Although \eqref{RPCA:ADAL-sym} can be applied directly to solve the RPCP problem, i.e., \eqref{prob:SPCP-missing-equiv} with $\delta=0$, an iteration complexity bound for this method is not known. However, when \eqref{RPCA:ADAL-sym} is applied to solve the problem given in \eqref{prob:SPCP-smooth} with $\delta=0$ where the two functions in the objective are both in $C^{1,1}$, we can prove an iteration complexity bound. That is, if we use \eqref{RPCA:ADAL-sym} to solve problem
\be \label{prob:SPCP-smooth-delta=0} \min_{L,S\in\mtn}\{F(L,S)\equiv f_\mu(L)+g_\nu(S): L+S=\pi_\Omega(D)\}, \ee
we have a provable complexity result. With $\|L\|_*$ replaced by $f_\mu(L)$ and $\xi\|S\|_1$ replaced by $g_\nu(S)$ in the augmented Lagrangian function $\LCal_\rho$ given in \eqref{RPCA:Lagrangian-function-missing}, and together with the two updating formulas for $\Lambda$ in \eqref{RPCA:ADAL-sym},  the optimality conditions for the two subproblems in \eqref{RPCA:ADAL-sym} yield:
\be
\label{lambda-gradient-relation}
\Lambda_{k+\half} = \nabla f_\mu(L_{k+1}) \quad \mbox{and} \quad \Lambda_{k+1} = \nabla g_\nu(S_{k+1}).
\ee
Let $\rho_k=\rho$ for all $k\geq 1$ for some $\rho>0$. By substituting \eqref{lambda-gradient-relation} into \eqref{RPCA:ADAL-sym}, and defining
\begin{align}
Q_{g}(L,S) &:= f_\mu(L) + g_\nu(S) - \langle \nabla g_\nu(S), L + S-\pi_\Omega(D) \rangle + \frac{1}{2\rho}\|L+S-\pi_\Omega(D)\|_F^2, \label{def:Q_g}\\
Q_{f}(L,S) &:= f_\mu(L) - \langle \nabla f_\mu(L), L + S-\pi_\Omega(D) \rangle + \frac{1}{2\rho}\|L+S-\pi_\Omega(D)\|_F^2 + g_\nu(S), \label{def:Q_f}
\end{align}
we obtain the alternating linearization method given in \textbf{Algorithm~\ref{alg:ALM-general}} that was proposed and analyzed in \cite{Goldfarb-Ma-Scheinberg-2010} by Goldfarb \etal for minimizing the sum of two convex functions. The ALM method, with a nonincreasing sequence $\{\rho_k\}_{k\in\integers_+}$ of proximal term parameters, was first proposed by Kiwiel \etal \cite{Kiwiel-Rosa-1999}. However, no iteration complexity analysis was presented in \cite{Kiwiel-Rosa-1999}.
\begin{algorithm}[h!]
    \caption{Alternating Linearization Method~(ALM)}\label{alg:ALM-general}
    {\small
    \begin{algorithmic}[1]
    \STATE \textbf{input:} $L_0\in\reals^{m\times n}$, $\rho>0$
    \STATE $k\gets 0$, $S_0\gets \pi_\Omega(D)-L_0$
    \WHILE{not converged}
        \STATE $L_{k+1} \gets \argmin_L Q_{g}(L,S_k)$
        \STATE $S_{k+1} \gets \argmin_S Q_{f}(L_{k+1},S)$
        \STATE $k\gets k+1$
    \ENDWHILE
    \RETURN $(L_k,S_k)$
    \end{algorithmic}
    }
\end{algorithm}

In \textbf{Algorithm~\ref{alg:ALM-general}}, the functions $f_\mu$ and $g_\nu$ are alternatingly replaced by their linearizations plus a proximal regularization term to get an approximation to the original function $F$. Thus, \textbf{Algorithm~\ref{alg:ALM-general}} can also be viewed as a proximal point algorithm. 

\begin{theorem}\label{the:ALM}
Suppose $(L^*,S^*)$ is an optimal solution to \eqref{prob:SPCP-smooth-delta=0} with $\delta=0$. The sequence $\{(L_k,S_k)\}_{k\in\integers_+}$ generated by \textbf{Algorithm~\ref{alg:ALM-general}} with $\rho \leq \min\{\mu,\nu\}$ satisfies:
\begin{align}\label{the:ALM-inequa} F(\pi_\Omega(D)-S_k,S_k) - F(L^*,S^*) \leq \frac{\|L^*-L_0\|_F^2}{4\rho k}.\end{align} Thus the sequence $\{F(\pi_\Omega(D)-S_k,S_k)\}$ produced by \textbf{Algorithm~\ref{alg:ALM-general}} converges to $F(L^*,S^*)$. Moreover, if $\beta\min\{\mu,\nu\} \leq\rho\leq \min\{\mu,\nu\}$ with $0<\beta\leq 1$, \textbf{Algorithm~\ref{alg:ALM-general}} needs $O(1/\epsilon^2)$ iterations to obtain an $\epsilon$-optimal solution to problem in \eqref{prob:SPCP-missing-equiv} with $\delta=0$.
\end{theorem}
\begin{proof}
For the proof see Corollary 2.4 in \cite{Goldfarb-Ma-Scheinberg-2010}.
\end{proof}

The iteration complexity bound of $\cO(1/\epsilon^2)$ in Theorem \ref{the:ALM} for RPCP problem can be improved further. The accelerated version of ALM proposed in \cite{Goldfarb-Ma-Scheinberg-2010} needs only $\cO(1/\epsilon)$ iterations to obtain an $\epsilon$-optimal solution, while the computational effort for the subproblems in each iteration is the same as that for the subproblems in ALM.

\subsection{Solving the Subproblems}
We now show how to solve the two subproblems in \textbf{Algorithm~\ref{alg:ALM-general}}. Note that the two subproblems at the iteration $k$ of \textbf{Algorithm~\ref{alg:ALM-general}} reduce to
\begin{align}
L_{k+1} &=\argmin_L f_\mu(L)
-\langle\nabla g_\nu(S_k),L+S_k-\pi_\Omega(D)\rangle+\frac{1}{2\rho}\|L+S_k-\pi_\Omega(D)\|_F^2, \label{sparse-low-rank-ALM-general-sub-1}\\
S_{k+1} &=\argmin_S 
-\langle\nabla f_\mu(L_{k+1}),L_{k+1}+S-\pi_\Omega(D)\rangle+\frac{1}{2\rho}\|L_{k+1}+S-\pi_\Omega(D)\|_F^2+g_\nu(S). \label{sparse-low-rank-ALM-general-sub-2}
\end{align}
Note that the first-order optimality conditions for \eqref{sparse-low-rank-ALM-general-sub-1} are \bea\label{sparse-low-rank-ALM-general-sub-1-optcond}W_\mu(L_{k+1})-Z_\nu(S_k)+\frac{1}{\rho}(L_{k+1}+S_k-\pi_\Omega(D))=0,\eea where $W_\mu(L)$ and $Z_\nu(S)$ are defined in \eqref{def:W-sigma(L)} and \eqref{def:Z-sigma(Y)}.
It is easy to check that
\begin{align}\label{sol:RPCA-L}
L_{k+1}=U\Diag(\sigma^*)V^\top, \quad \quad \sigma^*_i=\bar{\sigma}_i-\frac{\rho\bar{\sigma}_i}{\max\{\bar{\sigma}_i,\rho+\mu\}}
\quad for \quad i=1,\ldots, \min\{m,n\}
\end{align}
satisfies \eqref{sparse-low-rank-ALM-general-sub-1-optcond}, where
$U~\Diag(\bar{\sigma})V^\top$ is the singular value decomposition of the matrix $\rho Z_\nu(S_k)-S_k+\pi_\Omega(D)$. Thus, solving the subproblem \eqref{sparse-low-rank-ALM-general-sub-1} corresponds to a singular value decomposition. The first-order optimality conditions for \eqref{sparse-low-rank-ALM-general-sub-2} are: \bea\label{sparse-low-rank-ALM-general-sub-2-optcond}-W_\mu(L_{k+1})+\frac{1}{\rho}(L_{k+1}+S_{k+1}-\pi_\Omega(D))+Z_\nu(S_{k+1})=0.\eea
It is easy to check that
\begin{align}\label{sol:RPCA-Y}
(S_{k+1})_{ij}=\left\{
                 \begin{array}{ll}
                   \sgn(B_{ij})\max\left\{\frac{\nu}{\nu+\rho}|B_{ij}|,\ |B_{ij}|-\rho\xi\right\}, & \forall~(i,j)\in\Omega; \\
                   B_{ij}, & \forall~(i,j)\notin\Omega,
                 \end{array}
               \right.
\end{align}
satisfies \eqref{sparse-low-rank-ALM-general-sub-2-optcond}, where $B=\rho W_\mu(L_{k+1})-L_{k+1}+\pi_\Omega(D).$ Thus, solving the subproblem \eqref{sparse-low-rank-ALM-general-sub-2} can be done very cheaply.
\section{Numerical results for RPCP}
\label{sec:numerics-RPCP}
We conducted two sets of numerical experiments using \proc{ALM} (\textbf{Algorithm~\ref{alg:ALM-general}}) to solve \eqref{prob:SPCP-missing-equiv} with $\delta=0$. In the first set of experiments, we compared ALM with EADM and IADM~\cite{Ma09_1J} on problems with real data. In the second set, we solved randomly generated instances of the RPCP problem with missing observations. In this setting, we tested only the \proc{ALM} algorithm to see how the run times scaled with respect to problem parameters and size.
\subsection{RPCP problems with real data}
In this section, we report the numerical results of our ALM for solving \eqref{prob:RPCA-Nuclear-L1} with real data arising from background extraction from surveillance video, removing shadows and specularities from face images and video denoising from heavily corrupted data as in \cite{Candes-Li-Ma-Wright-RPCA-2009}. We will compare the performance of ALM with EADM and IADM in  \cite{Ma09_1J}.

EADM and IADM adopt a continuation strategy on $\rho$, in which, $\rho$ is updated in every iteration via \be \label{continuation-mu}\rho_{k+1} = \eta\rho_k \ee after updating the Lagrange multiplier matrix, where $\eta\in(0,1]$. The strategy used in EADM sets $\eta = 1/6$ and $\rho_0 = 2\|\sgn(D)\|$. The one used in IADM sets $\eta = 2/3$ and $\rho_0 = 0.8\|D\|$. To guarantee a fair comparison and to further accelerate ALM, we used the same continuation strategy as IADM, i.e., we set $\eta=2/3$ and $\rho_0 =0.8\|D\|$ in ALM.
Our ALM codes were written in MATLAB. The MATLAB codes of EADM and IADM were downloaded from http://watt.csl.illinois.edu/$\sim$perceive/matrix-rank/sample\_code.html. The default settings of EADM and IADM were used. 
All the numerical experiments were conducted on a Windows 7 machine with Intel Core i7-3520M Processor (4 MB cash, 2 cores at 2.9 GHz), and 16 GB RAM running  MATLAB 7.14 (64 bit).


Note that in each iteration of ALM, EADM and IADM, one has to compute a partial SVD with only certain number of leading singular values and corresponding singular vectors. To make a fair comparison, we adopted a modified version of LANSVD function of PROPACK with threshold option to compute the partial SVDs in ALM, EADM and IADM. This modified version of LANSVD function\footnote{Available from http://svt.stanford.edu/code.html} can compute only the singular values that are greater than a given threshold. In particular, in EADM and IADM, one has to solve a problem in the form of
$\min_L\LCal_{\rho_k}(L,S_k;\Lambda_k),$
which corresponds to computing SVD of $D+\rho_k\Lambda^k-S^k$ with only singular values that are greater than $\rho_k$. In ALM, one has to compute a partial SVD with threshold $\mu_k$ in order to compute \eqref{sol:RPCA-L}. 
Note that by setting a threshold for the SVD computation in ALM, $L_{k+1}$ in \eqref{sol:RPCA-L} is computed inexactly. However, as $\mu_k$ is approaching to 0, the computation of $L_{k+1}$ becomes more accurate.
We denote the total number of singular values computed during all the iterations by $\mathbf{lsv}$.


%
%
%

All algorithms were terminated when the relative infeasibility was less than $10^{-7}$, i.e.,
\[\frac{\|L_k+S_k-D\|_F}{\|D\|_F} < 10^{-7}.\]
\subsubsection{Foreground extraction from surveillance video}
\label{sec:video_test_results-RPCA}
Extracting the almost still background from a sequence of frames in a video is an important task in video surveillance. This problem is difficult due to the presence of moving foreground in the video. Interestingly, this problem can be formulated as a RPCP or SPCP problem depending on the existence of noise.
Note that by stacking the columns of each frame into a long vector, we can get a matrix $D$ whose columns correspond to the sequence of frames of the video. This matrix $D$ can be decomposed into the sum of three matrices $D:=L^0+S^0+N^0$, where $N^0$ is a dense noise matrix. The matrix $L^0$, which represents the backgrounds in the frames, should be of low rank due to the correlation between frames. The matrix $S^0$, which represents the moving foregrounds in the frames, should be sparse since the foreground usually occupies a small portion of each frame.

Here, we apply ALM to solve the RPCP problem \eqref{prob:RPCA-Nuclear-L1} corresponding to four videos introduced in \cite{Li04_1J}, where $N^0=\mathbf{0}$. Our first example is a sequence of 200 grayscale frames of size $144\times 176$ from a video of a hall at an airport; thus $D \in \br^{25344\times 200}$. The second example is a sequence of 300 grayscale frames of size $130\times 160$ from a video of an escalator at an airport; thus $D \in \br^{20800\times 300}$. In this video, the background is changing because of the moving escalator. The third example is a sequence of 250 grayscale frames of size $128\times 160$ from a video taken in a lobby; thus $D \in \br^{20480\times 250}$. In this video, the background changes with the switching off of some lights. The fourth example is a sequence of 320 grayscale frames of size $128\times 160$ from a video taken at a campus. In this video, the background is changing because of the oscillating trees.

Table \ref{tab:surveillance-video} summarizes the numerical results on these problems, where $\mathbf{cpu}$ denotes the running time of \proc{ALM} in seconds, and $\mathbf{lsv}$ denotes the average number of leading singular values computed. We note that the CPU times of ALM and were roughly comparable (the CPU time of ALM was slightly better in the third example, but IADM took slightly less CPU time in the first, second and fourth examples). However, both ALM and IADM were much faster than EADM. For example, ALM was about 15 times faster than EADM for the second test problem.
\begin{table}[h!]{\small
\begin{center}\caption{Comparison of ALM, EADM and IADM on surveillance video problems}\label{tab:surveillance-video} 
\begin{tabular}{|l c c | r r | r r | r r |}\hline
\multicolumn{3}{|c|}{} & \multicolumn{2}{|c|}{ALM} &
\multicolumn{2}{|c|}{Exact ADM} & \multicolumn{2}{|c|}{Inexact ADM} \\\hline

Problem & $m$ & $n$ & \textbf{lsv} & \textbf{cpu}   & \textbf{lsv} & \textbf{cpu} &  \textbf{lsv} & \textbf{cpu} \\\hline

Hall   & 25344 & 200 & 6713  &      43 & 104083 &     450 & 5357 &      38   \\\hline

Escalator & 20800 & 300  & 9513  &      68   & 168326 &     992  & 7392 &      59  \\\hline

Lobby  & 20480 & 250  & 6377  &      46  & 26385  &     855  & 3839 &      50  \\\hline

Campus & 20480 & 320  & 9596  &      70  & 181301 &     882  & 7889 &      63   \\\hline

\end{tabular}
\end{center}}
\end{table}

Some frames of the videos and the recovered backgrounds and foregrounds are shown in Figures \ref{fig:surveillance12} and \ref{fig:surveillance34}. We only show the frames produced by ALM, because EADM and IADM produced visually identical results. From these figures we can see that ALM can effectively separate the nearly still background from the moving foreground. We note that the numerical results in \cite{Candes-Li-Ma-Wright-RPCA-2009} show that the model \eqref{prob:RPCA-Nuclear-L1} produces much better results than other competing models for background separation in surveillance video.
\begin{figure}[h!]\vspace{-0.5cm}
\centering \subfigure{
\includegraphics[scale=0.51]{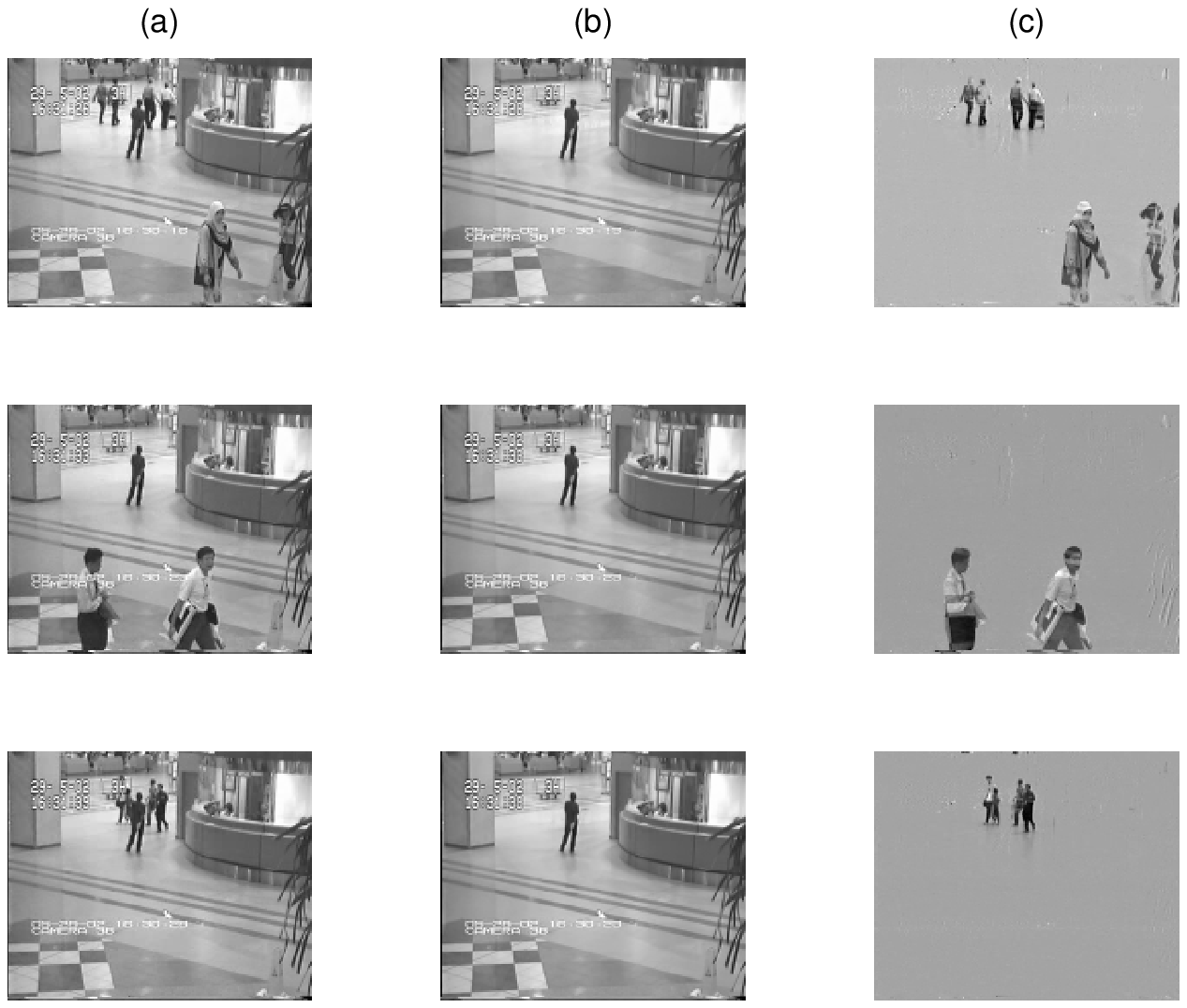}}
\centering \subfigure{
\includegraphics[scale=0.57]{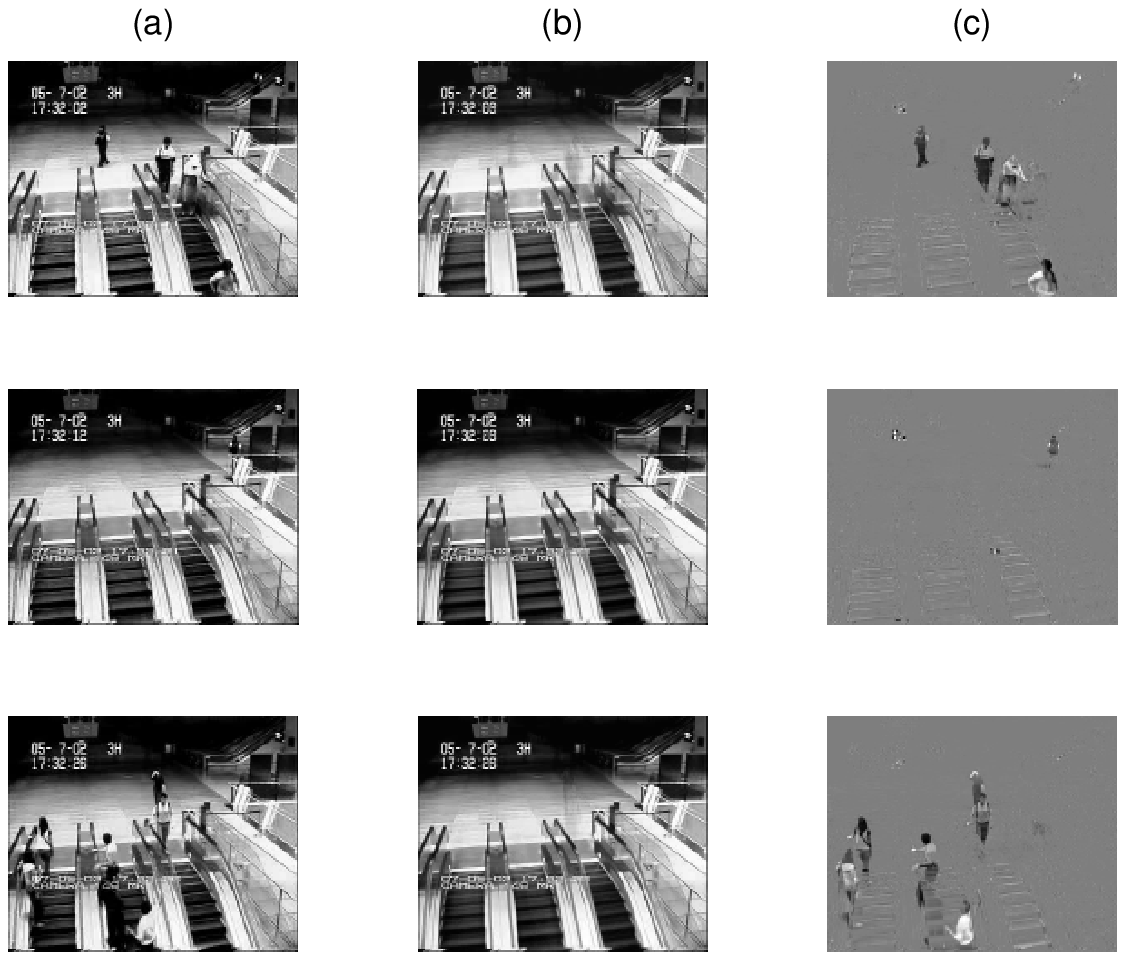}}
\caption{Images in columns (a) are the frames of the surveillance video. Images in columns (b) are the static backgrounds recovered by ALM. Note that the man who kept still in all the 200 frames of the first video stays as in the background. Images in columns (c) are the moving foregrounds recovered by ALM. Note that some artifacts of the escalator appears in the sparse matrix part because the background is changing as the escalator moves. }
\label{fig:surveillance12}
\end{figure}
\begin{figure}[h!]
\centering \subfigure{
\includegraphics[scale=0.57]{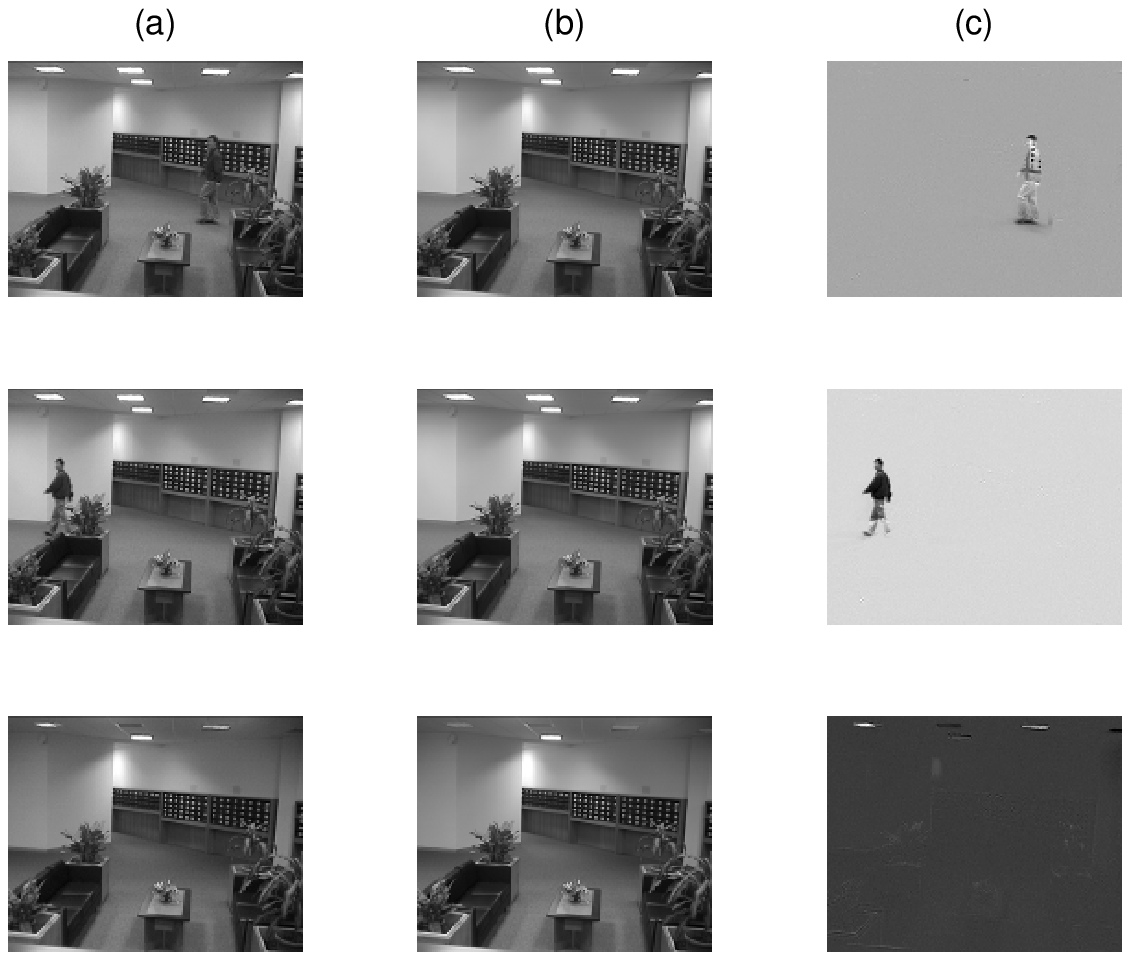}}
\centering \subfigure{
\includegraphics[scale=0.59]{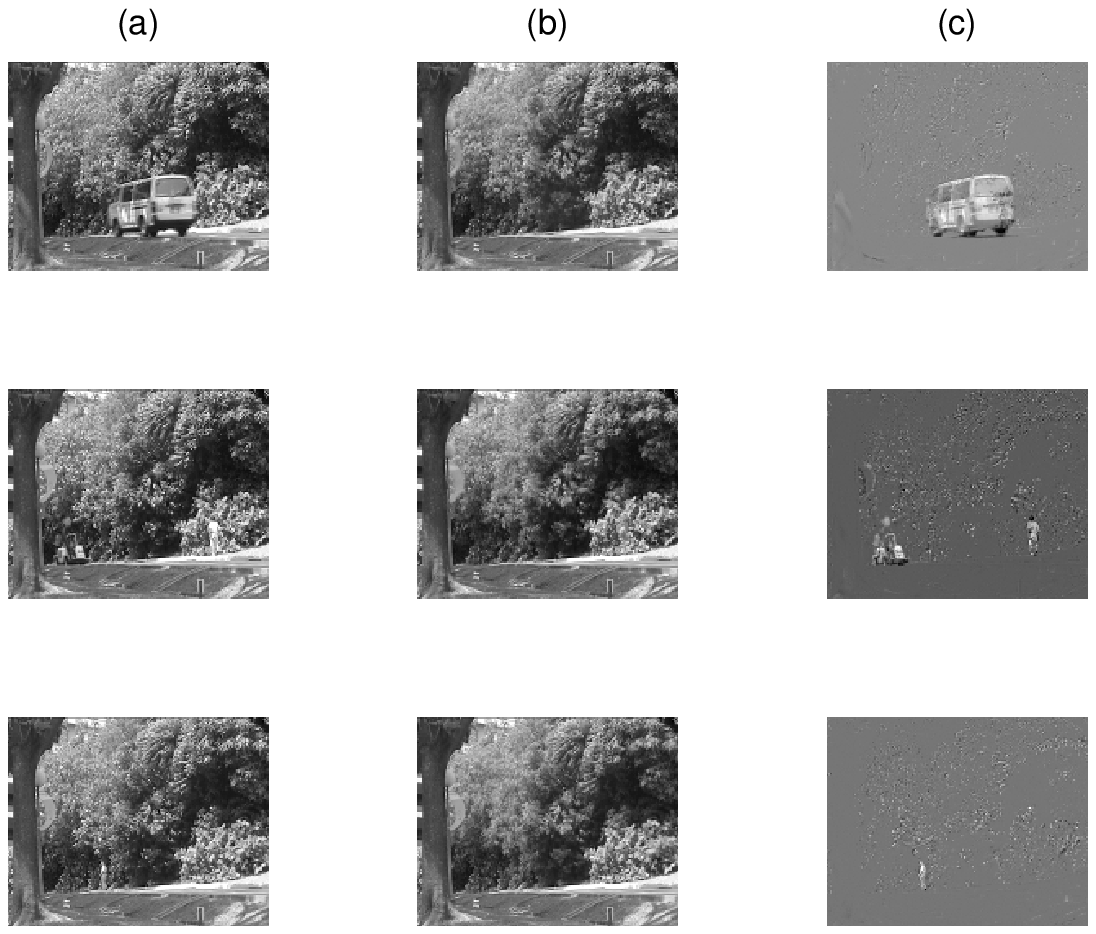}}
\caption{Images in columns (a) are the frames of the surveillance video. Images in columns (b) are the static backgrounds recovered by ALM. Images in columns (c) are the moving foregrounds recovered by ALM. Note that some artifacts of the trees appears in the sparse matrix part because the background is changing as the trees oscillate.}
\label{fig:surveillance34}
\end{figure}

%
%
%
%
%
%

\subsubsection{Removing shadows and specularities from face images}
\label{sec:shadow-removal}
Removing shadows and specularities from face images is another problem that fits into the RPCA framework when they are not corrupted by a dense noise matrix, i.e. $N^0=\mathbf{0}$. Suppose that we have many images of the same face taken under different illumination conditions. Hence, different shadows and specularities exist in different images. However, due to the correlation between the images, we should be able to remove shadows and specularities by RPCA. Note that by stacking the columns of each face image into a long vector, we get a matrix $D$ whose columns correspond to the sequence of images of the face.
This matrix $D$ can be decomposed into the sum of two matrices $D:=L^0+S^0$. The matrix $L^0$, which represents the face in the images, should be of low rank due to the correlation between images. The matrix $S^0$, which represents the shadows and specularities in the images, should be sparse since the shadows and specularities only occupy a small portion of each image.
\newpage
Here we present results on images taken from the dataset named ``yaleB01\_P00.tar.gz'' in the Yale B database \cite{Georghiades-Belhumeur-Kriegman-PAMI-2001}. Each image had a resolution $200\times 200$ and there were a total of 65 illuminations per subject, yielding a matrix $D\in\br^{40000\times 65}$. From Table \ref{tab:face} we see that ALM and IADM were roughly comparable in terms of the CPU time, and they were both about 20 times faster than EADM.
\begin{table}[ht]{\small
\begin{center}\caption{Comparison of ALM, EADM and IADM on face image problems}\label{tab:face} 
\begin{tabular}{|l c c | r r| r r| r r|}\hline
\multicolumn{3}{|c|}{} & \multicolumn{2}{|c|}{ALM} &
\multicolumn{2}{|c|}{Exact ADM} & \multicolumn{2}{|c|}{Inexact ADM} \\\hline

Problem & $m$ & $n$  & \textbf{lsv} & \textbf{cpu} &   \textbf{lsv} & \textbf{cpu} &  \textbf{lsv} & \textbf{cpu} \\\hline

Face image & 40000 & 65 & 2622  &      17 & 73721  &     315  & 2351 &      15 \\\hline

\end{tabular}
\end{center}}
\end{table}

Again, we only show the frames produced by ALM, because EADM and IADM produced visually identical results. The images recovered by ALM are shown in Figure \ref{fig:face_recognition}. From these figures we see that the shadows and specularities have been effectively detected and put into $S^0$ by ALM.
\begin{figure}[h!]\hspace{2cm}
\includegraphics[scale=0.7]{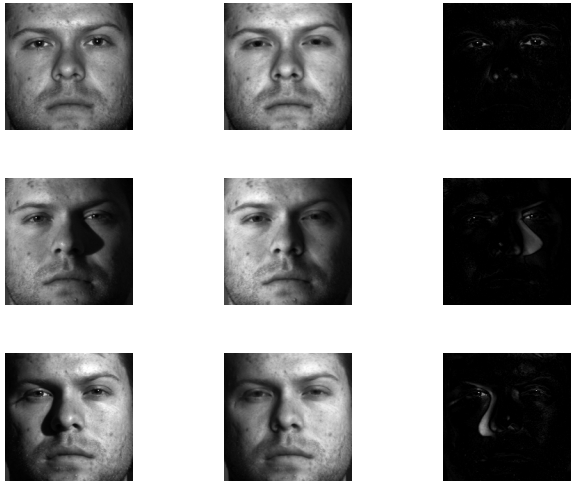}
\caption{The first column depicts the original face images. The second column depicts face images after removing shadows and specularities by ALM. The third column correspond to the removed shadows and specularities, i.e., the sparse errors.}
\label{fig:face_recognition}
\end{figure}
%
%
%
\subsubsection{Video denoising} We now consider the problem of denoising videos corrupted by \emph{impulsive noise} \cite{Yang-Zhang-Yin-08}. Again, as in section~\ref{sec:shadow-removal}, we assume that the frames are not corrupted by a dense noise matrix, i.e. $N^0=\mathbf{0}$. As before, we get a large matrix $D$ whose columns correspond to the frames of the video. Then $D$ can be decomposed into two parts $D:=L^0+S^0$, where the matrix $L^0$, which corresponds to the matrix formed by the uncorrupted video, is expected to be of low rank due to the correlation between frames; and the matrix $S^0$, which corresponds to the impulsive noise, is expected to be sparse. The colored video used in our experiment was downloaded from the website http://media.xiph.org/video/derf. This video consisted of 300 frames where each frame was an image stored in the RGB format, as a $144\times 176\times 3$ array. The video was then reshaped into a $144\times 176$ by $3\times 300$ matrix, i.e., $L^0\in\br^{25344\times 900}$. We then added impulsive noise $S^0$ to the video, by randomly choosing 20\% of the entries of $S^0$ and setting these entries to values drawn from an i.i.d. Gaussian distribution $\varrho\mathcal{N}(0,1)$ with $\varrho=10^3$. The other entries of $S^0$ were set to zero. Finally, we set $D:=L^0+S^0$.

From Table \ref{tab:video-denoising} we see that although ALM is much faster than EADM, it is substantially slower than IADM on this problem. Three frames of the video recovered by ALM are shown in Figure \ref{fig:video-denoising-akiyo}. From Figure \ref{fig:video-denoising-akiyo} we see that, ALM can denoise the video heavily corrupted by impulsive noise very well.

\begin{table}[h!]{\small
\begin{center}\caption{Comparison of ALM, EADM and IADM on video denoising problem}\label{tab:video-denoising} 
\begin{tabular}{|l c c | r r| r r| r r|}\hline
\multicolumn{3}{|c|}{} & \multicolumn{2}{|c|}{ALM} &
\multicolumn{2}{|c|}{Exact ADM} & \multicolumn{2}{|c|}{Inexact ADM} \\\hline
Problem & $m$ & $n$   & \textbf{lsv} & \textbf{cpu} &  \textbf{lsv} & \textbf{cpu} &   \textbf{lsv} & \textbf{cpu} \\\hline
Video denoising & 25344 & 900 & 34623 &     526  & 15896  &    1157  & 2023 &      249  \\\hline
\end{tabular}
\end{center}}
\end{table}

\begin{figure}[h!]\vspace{-1cm}
\centering \includegraphics[scale=0.6]{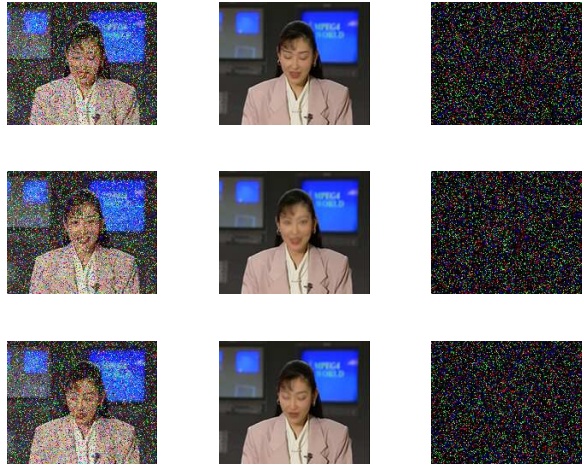} \vspace{-0.5cm}
\caption{Images in the first column are noisy frames with impulsive noise. Images in the second column are recovered frames. Images in the third column are recovered noise.} \label{fig:video-denoising-akiyo}
\end{figure}
\vspace{-0.5cm}
%
%
%

\subsection{Random RPCP Problems with Missing Data}
In this section, we solved randomly generated instances of the RPCP problem with missing data, i.e., \eqref{prob:SPCP-missing-equiv} with $\delta=0$. In this setting, we tested only the \proc{ALM} algorithm to see how the run times scaled with respect to problem parameters and size

We set $D:=L^0+S^0$, where the rank $r$ matrix $L^0\in\br^{n\times n}$ was created as the product $U V^\top$, of random matrices $U\in\br^{n\times r}$ and $V\in\br^{n\times r}$ with i.i.d. Gaussian entries $\mathcal{N}(0,1)$ and the sparse matrix $S^0$ was generated by choosing its support uniformly at random and its nonzero entries uniformly i.i.d. in the interval $\left[-\sqrt{8r/\pi},\sqrt{8r/\pi}\right]$. In our synthetic problems, we wanted the non-zero entries of the sparse component and the entries of the low-rank component to have the same magnitude. Indeed, for large $n$, $L_{ij}^0\sim \sqrt{r}~\mathcal{N}(0,1)$ for all i,j. Hence, $E[|L_{ij}^0|]=\sqrt{\frac{2r}{\pi}}$. Therefore, the way we created $S_{ij}^0$ for $(i,j)\in\Lambda$ ensures that $E[|S_{ij}^0|]=\sqrt{\frac{2r}{\pi}}$.

In Table \ref{tab:MC}, $\mathbf{c_r}:=\rank(L^0)/n$, $\mathbf{c_p}:=\|S^0\|_0/n^2$, the relative errors $\textbf{relL} := \|L-L^0\|_F / \|L^0\|_F$ and $\textbf{relS}: = \|S-S^0\|_F / \|S^0\|_F$, and the sampling ratio of $\Omega$, $\mathbf{SR}=m/n^2$. The $m$ indices in $\Omega$ were generated uniformly at random. We set $\xi=1/\sqrt{n}$ and stopped ALM when the relative infeasibility $\|L+S-\pi_\Omega(D)\|_F/\|\pi_\Omega(D)\|_F < 10^{-5}$ and for our continuation strategy, we set $\rho_0 = \|\pi_\Omega(D)\|_F/1.25$.


The test results obtained using ALM to solve \eqref{prob:SPCP-missing-equiv} with $\delta=0$ when the nonsmooth functions replaced by their smoothed approximations are given in Table \ref{tab:MC}. For each test scenario, i.e., for fixed $n$, $\mathbf{c_p}$, $\mathbf{c_p}$ and $\mathbf{SR}$ values, we solved $5$ random problem instances. All the results in Table \ref{tab:MC} were averaged over five runs. The CPU times reported were in seconds. From Table \ref{tab:MC} we see that ALM recovered the test matrices from a limited number of observations. Note that a fairly high number of samples was needed to obtain small relative errors due to the presence of noise. 
\begin{table}[ht]{\footnotesize
\begin{center}\caption{{Numerical results for noisy matrix completion problems}}\label{tab:MC}
\begin{tabular}{|c c| c c c c c|c c c c c|}\hline
$\mathbf{c_r}$ & $\mathbf{c_p}$    & \textbf{iter} & \textbf{lsv} & \textbf{relL} & \textbf{relS} & \textbf{cpu} & \textbf{iter} & \textbf{lsv} & \textbf{relL} & \textbf{relS} & \textbf{cpu}            \\\hline
\multicolumn{2}{|c|}{} & \multicolumn{5}{|c|}{$\mathbf{SR}=90\%, \mathbf{n} = 500$} & \multicolumn{5}{|c|}{$\mathbf{SR}=80\%, \mathbf{n}=500$} \\\hline

0.05 & 0.05 & 24 & 1811 & 5.4e-006 & 3.0e-005 & 3.9 & 25 & 2034 & 5.5e-006 & 2.9e-005 & 4.3 \\\hline
0.05 & 0.10 & 23 & 2516 & 8.7e-006 & 3.4e-005 & 5.5 & 25 & 2782 & 7.4e-006 & 2.7e-005 & 6.1  \\\hline
0.10 & 0.05 & 24 & 2702 & 8.2e-006 & 3.5e-005 & 5.4 & 29 & 7018 & 2.0e-003 & 9.0e-003 & 6.6  \\\hline
0.10 & 0.10 & 30 & 9900 & 4.2e-004 & 1.4e-003 & 6.3 & 29 & 11235 & 1.0e-002 & 3.2e-002 & 5.6  \\\hline

\multicolumn{2}{|c|}{} & \multicolumn{5}{|c|}{$\mathbf{SR}=90\%, \mathbf{n} = 1000$} & \multicolumn{5}{|c|}{$\mathbf{SR}=80\%, \mathbf{n}=1000$} \\\hline

0.05 & 0.05 & 26 & 3643 & 5.6e-006 & 3.3e-005 & 21.5 & 27 & 3530 & 6.6e-006 & 3.6e-005 & 21.5 \\\hline
0.05 & 0.10 & 27 & 4869 & 8.2e-006 & 3.1e-005 & 31.7 & 27 & 5326 & 7.3e-006 & 2.6e-005 & 34.2 \\\hline
0.10 & 0.05 & 26 & 5412 & 8.4e-006 & 3.7e-005 & 31.4 & 29 & 7470 & 5.2e-004 & 2.3e-003 & 47.1 \\\hline
0.10 & 0.10 & 25 & 7354 & 2.2e-005 & 7.0e-005 & 47.6 & 30 & 21495 & 6.3e-003 & 2.0e-002 & 37.3 \\\hline

\multicolumn{2}{|c|}{} & \multicolumn{5}{|c|}{$\mathbf{SR}=90\%, \mathbf{n} = 1500$} & \multicolumn{5}{|c|}{$\mathbf{SR}=80\%, \mathbf{n}=1500$} \\\hline

0.05 & 0.05 & 27 & 4692 & 6.9e-006 & 4.2e-005 & 49.7 & 27 & 5173 & 7.0e-006 & 3.8e-005 & 56.0  \\\hline
0.05 & 0.10 & 27 & 7427 & 9.1e-006 & 3.5e-005 & 84.8 & 28 & 7930 & 7.0e-006 & 2.6e-005 & 91.4 \\\hline
0.10 & 0.05 & 28 & 7859 & 9.0e-006 & 4.0e-005 & 83.3 & 28 & 10212 & 1.7e-004 & 7.5e-004 & 117.2 \\\hline
0.10 & 0.10 & 27 & 11451 & 9.1e-006 & 2.9e-005 & 143.2 & 31 & 33003 & 4.7e-003 & 1.4e-002 & 113.7 \\\hline
\end{tabular}
\end{center}}
\end{table}
\section{Proximal Gradient Method for SPCP}
\label{sec:SPCP}
In this section we focus on the problem \eqref{prob:SPCP-partial-smooth} to solve the SPCP problem \eqref{prob:SPCP-missing-equiv} and show how the proximal gradient algorithm, i.e., the FISTA algorithm in \cite{Beck-Teboulle-2009}, can be applied to it. An advantage of this algorithm over to ALM algorithm is that only one of the two terms in the objective function of the SPCP problem \eqref{prob:SPCP-missing-equiv} needs to be smoothed. Note that this algorithm can also easily be applied to the RPCP problem. The FISTA algorithm for \eqref{prob:SPCP-partial-smooth} is given in \textbf{Algorithm~\ref{alg:nsNesterov}}.
\begin{algorithm}[h!]
    \caption{Partially Smooth Proximal Gradient~(PSPG)}\label{alg:nsNesterov}
    {\small
    \begin{algorithmic}[1]
    \STATE \textbf{input:} $L_0\in\reals^{m\times n}$, $Y_0\in\reals^{m\times n}$
    \STATE $k\gets 0$
    \WHILE{$k\leq k^*$}
    \STATE $(L_k,S_k) \gets \argmin_{L,S} \left\{\xi\|\pi_\Omega(S)\|_1+ \left\langle \nabla f_\mu(Y_k),L\right\rangle + \frac{1}{2\mu}\|L-Y_k\|_F^2: (L,S)\in\chi\right\}$\label{algeq:pspg-subproblem}
    \STATE $t_{k+1}\gets (1+\sqrt{1+4t_k^2})/2$
    \STATE $Y_{k+1} \gets L_k+\frac{t_k-1}{t_{k+1}}(L_k-L_{k-1})$
    \STATE $k \gets k + 1$
    \ENDWHILE
    \RETURN $(L_{k^*},S_{k^*})$
    \end{algorithmic}
    }
\end{algorithm}

Mimicking the proof in \cite{Beck-Teboulle-2009}, it is easy to show that \textbf{Algorithm~\ref{alg:nsNesterov}} converges to the optimal solution of \eqref{prob:SPCP-partial-smooth}. Given $(L_0,S_0)\in\chi:=\{(L,S):~\norm{L+S-\pi_\Omega(D)}_F\leq\delta\}$, e.g., $L_0=\mathbf{0}$ and $S_0=\pi_\Omega(D)$, the current algorithm keeps all iterates in $\chi$,
and hence it enjoys the full convergence rate of $\cO\left(\frac{\mu^{-1}}{k^2}\right)$. Thus, setting $\mu=\Omega(\epsilon)$, \textbf{Algorithm~\ref{alg:nsNesterov}} computes an $\epsilon$-optimal, feasible solution of problem \eqref{prob:SPCP-missing-equiv} in $k^*=\cO(1/\epsilon)$ iterations. However, overall complexity depends on per-iteration complexity, i.e. complexity of solving subproblem in line~\ref{algeq:pspg-subproblem} of \textbf{Algorithm~\ref{alg:nsNesterov}}. Since the solution and the complexity of this step was not analyzed before, FISTA~\cite{Beck-Teboulle-2009} has not been applied to solve \eqref{prob:SPCP-partial-smooth} previously and its overall complexity  for \eqref{prob:SPCP-partial-smooth} has been unknown.

In this paper, we show that the optimization subproblems in \textbf{Algorithm~\ref{alg:nsNesterov}} can be solved efficiently. The subproblem that has to be solved at each iteration to compute $(L_k, S_k)$ has the following generic form:
\begin{align}
\label{eq:subproblem_nsa}
(P_{ns}):\  \min\left\{\xi\norm{\pi_\Omega(S)}_1+\left\langle Q, L-\tilde{L}\right\rangle+\frac{1}{2\rho}\norm{L-\tilde{L}}_F^2:\ (L,S)\in \chi\right\},
\end{align}
for some $\rho>0$. Lemma~\ref{lem:subproblem} shows that these computations can be done efficiently.
\newpage
\begin{lemma}
\label{lem:subproblem}
The optimal solution $(L^*,S^*)$ to problem $(P_{ns})$ can be written in closed form as follows.

When $\delta>0$,
\begin{align}
&S^*=\sgn\left(\pi_\Omega\left(D-q(\tilde{L})\right)\right)\odot\max\left\{|\pi_\Omega\left(D-q(\tilde{L})\right)|-\xi\frac{(1+\rho\theta^*)}{\theta^*}~E,\ \mathbf{0}\right\}-\pi_{\Omega^c}\left(q(\tilde{L})\right), \label{lemeq:S}\\
&L^*= \pi_\Omega\left(\frac{\rho\theta^*}{1+\rho\theta^*}~(D-S^*)+\frac{1}{1+\rho\theta^*}~q(\tilde{L})\right)+\pi_{\Omega^c}\left(q(\tilde{L})\right), \label{lemeq:L}
\end{align}
where $q(\tilde{L}):=\tilde{L}-\rho~Q$, $E$ and $\mathbf{0}\in\reals^{m\times n}$ are matrices with all components equal to ones and zeros, respectively, and $\odot$ denotes the componentwise multiplication operator. $\theta^*=0$ if $\norm{\pi_\Omega(D-q(\tilde{L}))}_F\leq\delta$; otherwise, $\theta^*$ is the unique positive solution of the nonlinear equation $\phi(\theta)=\delta$, where
\begin{align}
\phi(\theta):= \left\|\min\left\{\frac{\xi}{\theta}~E,\ \frac{1}{1+\rho\theta}~|\pi_\Omega(D-q(\tilde{L}))|\right\}\right\|_F.
\end{align}
Moreover, $\theta^*$ can be efficiently computed in $\cO(|\Omega|\log(|\Omega|))$ time.

When $\delta=0$,
\begin{equation}
\label{lemeq:XS_nonsmooth_delta0}
S^*=\sgn\left(\pi_\Omega(D-q(\tilde{L}))\right)\odot\max\left\{|\pi_\Omega(D-q(\tilde{L}))|-\xi\rho~E,\ \mathbf{0}\right\}-\pi_{\Omega^c}\left(q(\tilde{L})\right),
\end{equation}
and $L^*=\pi_\Omega(D)-S^*$.
\end{lemma}
\begin{proof}
See Appendix~\ref{sec:appendix2} for the proof.
\end{proof}

We establish in Lemma~\ref{lem:subproblem} that, provided with $q(\tilde{L}):=\tilde{L}-\rho Q$, the generic form subproblem $(P_{ns})$ in \eqref{eq:subproblem_nsa} can be solved efficiently. In line~\ref{algeq:pspg-subproblem} of PSPG, i.e. \textbf{Algorithm~\ref{alg:nsNesterov}}, $(L_k,S_k)$ is computed by solving $(P_{ns})$ with $\tilde{L}=Y_k$, $Q=\grad f_\mu(Y_k)$ and $\rho=\mu$. Hence, before using Lemma~\ref{lem:subproblem} to compute $(L_k,S_k)$, we have to compute $q(Y_k)=Y_k-\mu~\grad f_\mu(Y_k)$. Note that $\grad f_\mu(Y_k)=W_\mu(Y_k)=U~\Diag\left(\min\left\{\frac{\bar{\sigma}}{\mu},1\right\}\right)V^\top$, where $Y_k=U~\Diag(\bar{\sigma})V^\top$. However, computing $\grad f_\mu(Y_k)$ first and then computing $q(Y_k)$ can cause numerical problems. It is easy to check that
\begin{align}
\label{eq:qyk}
q(Y_k)=U\Diag\left(\left(\bar{\sigma}-\mu\right)_+\right)V^\top.
\end{align}
Hence, one can compute $q(Y_k)$ directly without computing $\grad f_\mu(Y_k)$ to improve numerical stability.

\section{Numerical results for SPCP}
\label{sec:numerics-SPCP}
We conducted two sets of numerical experiments using \proc{PSPG} (\textbf{Algorithm~\ref{alg:nsNesterov}}) to solve ~\eqref{prob:SPCP}, where $\xi=\frac{1}{\sqrt{\max\{m,n\}}}$. In the first set we solved randomly generated instances of the SPCP problem. In this setting, first we tested only the \proc{PSPG} algorithm to see how the run times scaled with respect to problem parameters and size; then we compared the \proc{PSPG} algorithm with an alternating direction augmented Lagrangian algorithm ASALM~\cite{Tao09_1J}. In the second set of experiments, we ran the PSPG and ASALM algorithms to extract foreground from an airport security noisy video~\cite{Li04_1J}. The MATLAB code for  \proc{PSPG} is available
at~\url{http://www2.ie.psu.edu/aybat/codes.html} and the code for ASALM
is available on request from the authors of~\cite{Tao09_1J}. All the numerical experiments were conducted on a Windows 7 machine with Intel Core i7-3520M Processor (4 MB cash, 2 cores at 2.9 GHz), and 16 GB RAM running  MATLAB 7.14 (64 bit).
\subsection{Implementation details}
Note that in each iteration of PSPG, in order to solve the subproblem in line~\ref{algeq:pspg-subproblem} of \textbf{Algorithm~\ref{alg:nsNesterov}}, one has to perform a matrix shrinkage operation to compute $q(Y_k)$ as given in \eqref{eq:qyk}, which corresponds to computing the SVD of $Y_k$ and is expensive for large matrices. However, we do not have to compute the whole SVD, as only the singular values that are larger than the smoothing parameter $\mu$ and the corresponding singular vectors are needed. As in section~\ref{sec:numerics-RPCP}, during each partial SVD, we used the modified version of PROPACK \cite{Larsen-Propack} to compute only those singular values that are greater than the threshold $\mu$ and the corresponding singular vectors. 

Finally, when we implemented PSPG (\textbf{Algorithm~\ref{alg:nsNesterov}}), we adopted a continuation strategy on $\mu$, in which, $\mu$ is updated via \be \label{continuation-mu-pspg}\mu_{k+1} = \left\{
                                     \begin{array}{ll}
                                       \eta\mu_k, & k\leq \bar{K}; \\
                                       \mu_k, & k>\bar{K},
                                     \end{array}
                                   \right.\ee
after the subproblem in line~\ref{algeq:pspg-subproblem} of \textbf{Algorithm~\ref{alg:nsNesterov}} is solved. The strategy used in PSPG sets $\bar{K}=30$, $\eta = 2/3$ and $\mu_0 = 0.8\|D\|$. Note that the theoretical convergence properties of PSPG are not significantly affected by the continuation strategy on $\mu$ and are still valid after the first $\bar{K}$ iterations. On the other hand, in our numerical experiments we noticed that the continuation strategy significantly decreases the run time of the algorithm as the number of leading singular value computations decreases for larger values of $\mu$ -see the computation of $q(Y_k)$ in \eqref{eq:qyk} for the effect of $\mu$.
\subsection{Random Stable Principle Component Pursuit Problems}
\label{sec:rPCA_results}
\subsubsection{Experimental Setting}
We tested \proc{PSPG} on randomly generated stable principle component pursuit problems. The data matrices for these problems, $D=L^0+S^0+N^0$, were generated as follows
\begin{enumerate}[i.]
\item $L^0=UV^T$, such that $U\in\reals^{n\times r}$, $V\in\reals^{n\times r}$ for $r=c_r n$ and $U_{ij}\sim \mathcal{N}(0,1)$, $V_{ij}\sim
  \mathcal{N}(0,1)$ for all $i,j$ are independent standard Gaussian variables and $c_r\in\{0.05, 0.1\}$,
\item $\Lambda\subset\{(i,j):\ 1 \leq i,j\leq n\}$ was chosen uniformly at random such that its cardinality $|\Lambda|=p$ for $p=c_p n^2$ and $c_p\in\{0.05, 0.1\}$,
\item $S^0_{ij}\sim\mathcal{U}\left[-\sqrt{\frac{8r}{\pi}},\sqrt{\frac{8r}{\pi}}\right]$ for all $(i,j)\in\Lambda$ are independent uniform random variables, 
\item $N^0_{ij}\sim \varrho~\mathcal{N}(0,1)$ for all $i,j$ are independent Gaussian variables.
\end{enumerate}

We created 10 random problems of size $n\in\{500, 1000, 1500\}$, i.e. $D\in\reals^{n
  \times n}$, for each of the two choices of $c_r$ and $c_p$ using the procedure described above, where $\varrho$ was set such that the signal-to-noise ratio of $D$ is either $80$dB or $45$dB. The signal-to-noise ratio of $D$ is given by
  \begin{align}
  \label{eq:snr}
  \rm{SNR}(D)=10\log_{10}\left(\frac{E\left[\norm{L^0+S^0}_F^2\right]}{E\left[\norm{N^0}_F^2\right]}\right)=10\log_{10}\left(\frac{c_r n+c_s \frac{8r}{3\pi}}{\varrho^2}\right).
  \end{align}
  Hence, for a given SNR value, we selected $\varrho$ according to \eqref{eq:snr}. Table~\ref{tab:snr} displays the $\varrho$ values that we used in our experiments.
  \begin{table}[!htb]
    \begin{adjustwidth}{-2em}{-2em}
    \centering
    \caption{$\varrho$ values depending on the experimental setting}
    \renewcommand{\arraystretch}{1.1}
    {\scriptsize
    \begin{tabular}{|c|c|c|c|c|c|c|}
    \hline
    SNR&n&$\mathbf{c_r}$=\textbf{0.05} $\mathbf{c_p}$=\textbf{0.05}&$\mathbf{c_r}$=\textbf{0.05} $\mathbf{c_p}$=\textbf{0.1}&$\mathbf{c_r}$=\textbf{0.1} $\mathbf{c_p}$=\textbf{0.05}&$\mathbf{c_r}$=\textbf{0.1} $\mathbf{c_p}$=\textbf{0.1}\\\hline
    \multirow{3}{*}{$80\mathrm{dB}$}
    &$\mathbf{500}$
    & 0.5E-03 & 0.5E-03 & 0.7E-03 & 0.7E-03 \\ \cline{2-6}
    &$\mathbf{1000}$
    & 0.7E-03 & 0.7E-03 & 1.0E-03 & 1.0E-03 \\ \cline{2-6}
    &$\mathbf{1500}$
    & 0.9E-03 & 0.9E-03 & 1.3E-03 & 1.3E-03 \\ \hline
    \multirow{3}{*}{$45\mathrm{dB}$}
    &$\mathbf{500}$
    & 2.9E-02 & 2.9E-02 & 4.1E-02 & 4.1E-02 \\ \cline{2-6}
    &$\mathbf{1000}$
    & 4.1E-02 & 4.1E-02 & 5.7E-02 & 5.9E-02 \\ \cline{2-6}
    &$\mathbf{1500}$
    & 5.0E-02 & 5.1E-02 & 7.0E-02 & 7.1E-02 \\ \hline
    \end{tabular}
    \label{tab:snr}
    }
    \end{adjustwidth}
\end{table}

As in \cite{Tao09_1J}, we set $\delta = \sqrt{(n + \sqrt{8n})}\varrho$ in \eqref{prob:SPCP} in the first set of experiments for both \proc{PSPG} and \proc{ASALM}.
We terminated PSPG when
\begin{align}
\label{eq:stopping_cond}
\frac{\norm{(L_{k+1},S_{k+1})-(L_{k},S_{k})}_F}{\norm{(L_{k},S_{k})}_F+1}\leq 0.05~\varrho,
\end{align}
where all the components of the dense noise matrix, $N^0$, are i.i.d. with standard deviation $\varrho$, i.e. $N^0_{ij}\sim \varrho\mathcal{N}(0,1)$ for all $i=1,\ldots,m$ and $j=1,\ldots,n$. On the other hand, we terminated \proc{ASALM} either when it computed a solution with better relative errors compared to the \proc{PSPG} solution for the same problem or when an iterate satisfied \eqref{eq:stopping_cond}. To be more specific, let $D=L^0+S^0+N^0$ be generated as discussed above. For a given iterate $(\tilde{L},\tilde{S})$, we denote its relative error with $\mathbf{relL}(\tilde{L})=\norm{\tilde{L}-L^0}_F/\norm{L^0}_F$ and $\mathbf{relS}(\tilde{S})=\norm{\tilde{S}-S^0}_F/\norm{S^0}_F$. Suppose that the relative errors of the low-rank and sparse components of the PSPG solution are $r_L$ and $r_S$, respectively. If $(k+1)$-th ASALM iterate, $(L_{k+1},S_{k+1})$, is the first one to satisfy the condition
\begin{align}
\label{eq:stopping_cond_2}
\mathbf{relL}(L_{k+1})\leq r_L \quad \hbox{and} \quad \mathbf{relS}(S_{k+1})\leq r_S,
\end{align}
before the condition given in \eqref{eq:stopping_cond} holds, then we chose ASALM to return $(L_k,S_k)$ to be fair in terms of run time comparison, i.e., all the run times reported for ASALM are the cpu times required for \eqref{eq:stopping_cond} or \eqref{eq:stopping_cond_2} to hold, whichever comes first. 

\begin{figure} [b!]
    \centering
    \includegraphics[scale=0.6]{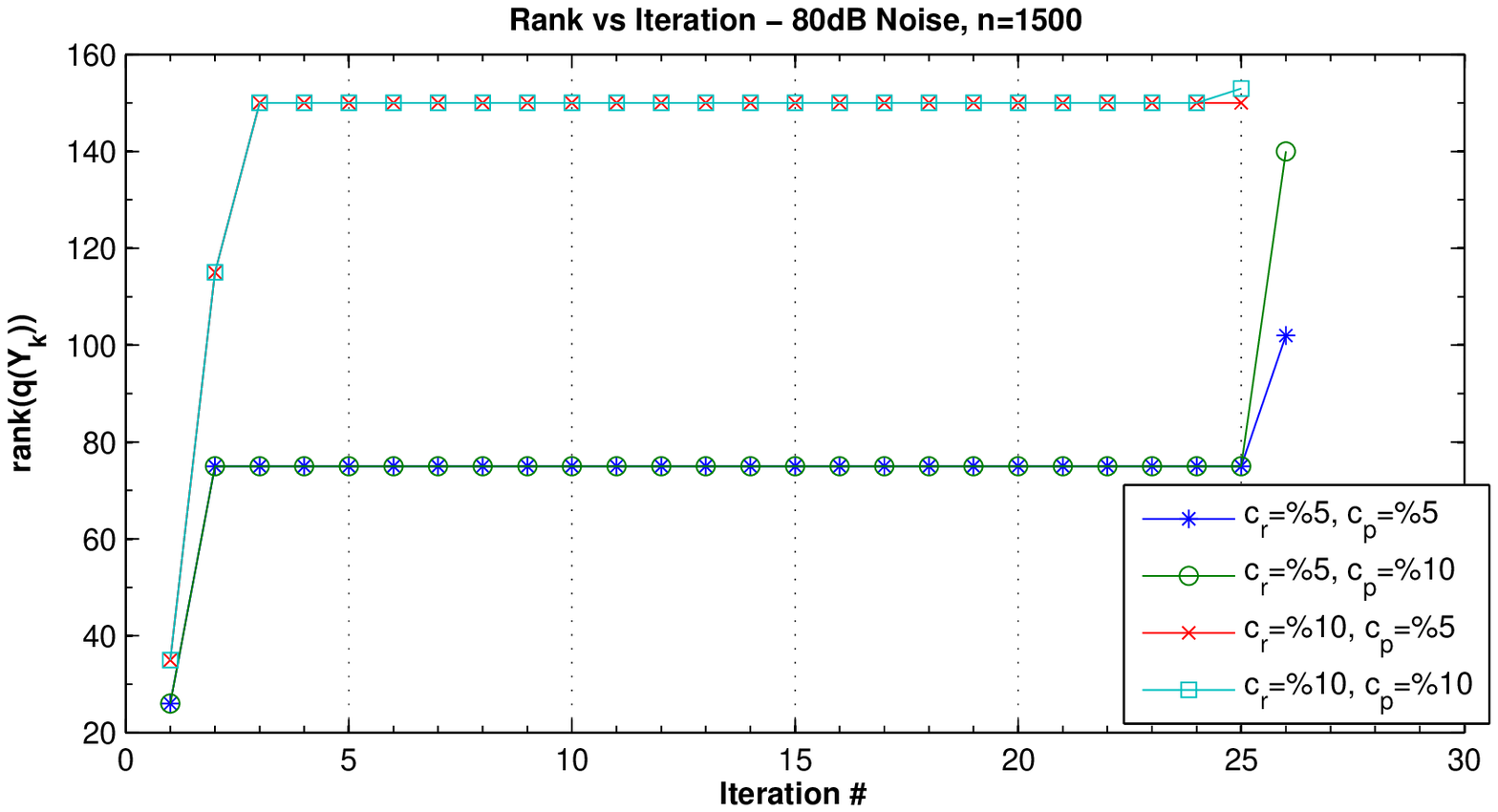}
    \includegraphics[scale=0.6]{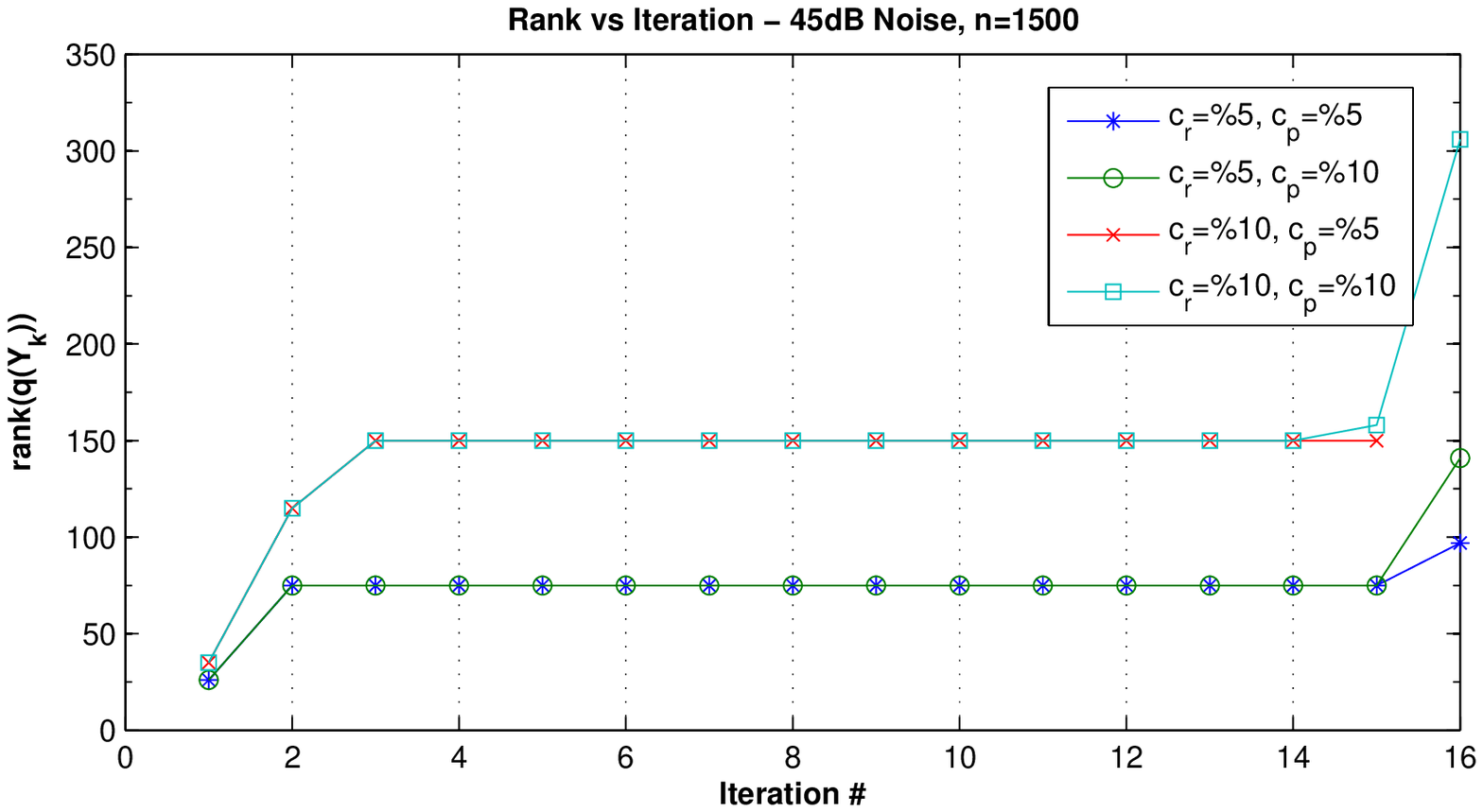}
    \caption{Change of $\rank(q(Y_k))$ during \proc{PSPG} iterations for $D\in\reals^{n\times n}$ and n=1500}
    \label{fig:rank_pspg}
\end{figure}
In each iteration of ASALM, SVD of a low-rank matrix has to be computed as in PSPG. However, the ASALM code, provided by the authors of~\cite{Tao09_1J}, does this computation very inefficiently. Therefore, in order to compare both codes on the same grounds, we used modified LANSVD function of PROPACK with threshold option in PSPG and ASALM to compute low-rank SVDs more efficiently (modified LANSVD function with threshold option can be downloaded from http://svt.stanford.edu/code.html).


\subsubsection{Numerical results}
As discussed earlier, solving subproblem in line~\ref{algeq:pspg-subproblem} of \proc{PSPG} (\textbf{Algorithm~\ref{alg:nsNesterov}}) is the only bottleneck operation and according to Lemma~\ref{lem:subproblem}, it can be computed very efficiently when $q(Y_k)$ is given. Hence, effectiveness of  \proc{PSPG} depends on the rank of $q(Y_k)$ to be small -- see \eqref{eq:qyk}. It can be seen in Figure~\ref{fig:rank_pspg} that $\rank(q(Y_k))\leq\rank(L^0)$ except in the last one or two iterations and note that $L^0$ is a low-rank matrix. As shown in Figure~\ref{fig:rank_pspg} for both $45$dB and $80$dB, $\rank(q(Y_k))$ increases and stabilizes exactly at $\rank(L^0)$. This behavior plotted in the figure is true for all test problems we solved.
\begin{table}[!htb]
    \begin{adjustwidth}{-2em}{-2em}
    \centering
    \caption{PSPG: Solution time for decomposing $D\in\reals^{n\times n}$, $n\in\{500, 1000, 1500\}$}
    \renewcommand{\arraystretch}{1.3}
    {\scriptsize 
    \begin{tabular}{ccc|c|c|c|c|}
    \cline{4-7}
    &&&$\mathbf{c_r}$=\textbf{0.05} $\mathbf{c_p}$=\textbf{0.05}&$\mathbf{c_r}$=\textbf{0.05} $\mathbf{c_p}$=\textbf{0.1}&$\mathbf{c_r}$=\textbf{0.1} $\mathbf{c_p}$=\textbf{0.05}&$\mathbf{c_r}$=\textbf{0.1} $\mathbf{c_p}$=\textbf{0.1}\\\hline
    \multicolumn{1}{|c|}{SNR}&\multicolumn{1}{|c|}{n}&Field& \textbf{avg}~/~max & \textbf{avg}~/~max & \textbf{avg}~/~max & \textbf{avg}~/~max\\ \hline
    \multicolumn{1}{|c|}{\multirow{9}{*}{$80\mathrm{dB}$}}
    &\multicolumn{1}{|c|}{\multirow{3}{*}{$\mathbf{500}$}}
    & $\mathbf{svd}$ & 28~/~28&28~/~28&26~/~26&27~/~27 \\
    \multicolumn{1}{|c|}{}&\multicolumn{1}{|c|}{}& $\mathbf{lsv}$ & 1152.6~/~1163.0&1183.0~/~1191.0&1482.8~/~1502.0&1656.1~/~1680.0\\
    \multicolumn{1}{|c|}{}&\multicolumn{1}{|c|}{}& $\mathbf{cpu}$ &
    4.7~/~4.8&5.4~/~5.5&4.0~/~4.2&5.2~/~5.3 \\ \cline{2-7}
    \multicolumn{1}{|c|}{}
    &\multicolumn{1}{|c|}{\multirow{3}{*}{$\mathbf{1000}$}}
    & $\mathbf{svd}$ & 26.3~/~27&27~/~27&25~/~25&26~/~26 \\
    \multicolumn{1}{|c|}{}&\multicolumn{1}{|c|}{}& $\mathbf{lsv}$ & 1557.0~/~1714.0&1752.0~/~1778.0&2525.1~/~2530.0&2730.7~/~2738.0\\
    \multicolumn{1}{|c|}{}&\multicolumn{1}{|c|}{}& $\mathbf{cpu}$ & 14.4~/~16.5&18.5~/~18.8&13.6~/~13.8&17.6~/~17.8\\ \cline{2-7}
    \multicolumn{1}{|c|}{}
    &\multicolumn{1}{|c|}{\multirow{3}{*}{$\mathbf{1500}$}}
    & $\mathbf{svd}$ & 26~/~26&27~/~27&25~/~25&25~/~25\\
    \multicolumn{1}{|c|}{}&\multicolumn{1}{|c|}{}& $\mathbf{lsv}$ & 2084.8~/~2094.0&2474.4~/~2492.0&3685.1~/~3695.0&3690.3~/~3703.0\\
    \multicolumn{1}{|c|}{}&\multicolumn{1}{|c|}{}& $\mathbf{cpu}$ & 28.7~/~32.8&41.3~/~45.2&34.0~/~34.3&38.2~/~38.6\\ \hline
    \multicolumn{1}{|c|}{\multirow{9}{*}{$45\mathrm{dB}$}}
    &\multicolumn{1}{|c|}{\multirow{3}{*}{$\mathbf{500}$}}
    & $\mathbf{svd}$ & 18~/~18&18~/~18&16~/~16&17~/~17 \\
    \multicolumn{1}{|c|}{}&\multicolumn{1}{|c|}{}& $\mathbf{lsv}$ & 806.0~/~819.0&839.8~/~852.0&901.8~/~922.0&1080.5~/~1102.0\\
    \multicolumn{1}{|c|}{}&\multicolumn{1}{|c|}{}& $\mathbf{cpu}$ &
    3.1~/~3.3&3.5~/~3.6&2.4~/~2.6&3.3~/~3.4 \\ \cline{2-7}
    \multicolumn{1}{|c|}{}
    &\multicolumn{1}{|c|}{\multirow{3}{*}{$\mathbf{1000}$}}
    & $\mathbf{svd}$ & 16.9~/~17&17~/~17&15~/~15&16~/~16 \\
    \multicolumn{1}{|c|}{}&\multicolumn{1}{|c|}{}& $\mathbf{lsv}$ & 1100.8~/~1146.0&1172.5~/~1198.0&1475.0~/~1480.0&1685.7~/~1702.0\\
    \multicolumn{1}{|c|}{}&\multicolumn{1}{|c|}{}& $\mathbf{cpu}$ &
    10.0~/~10.6&11.6~/~11.8&7.9~/~8.1&11.0~/~11.4\\ \cline{2-7}
    \multicolumn{1}{|c|}{}
    &\multicolumn{1}{|c|}{\multirow{3}{*}{$\mathbf{1500}$}}
    & $\mathbf{svd}$ & 16~/~16&17~/~17&15~/~15&15.6~/~16\\
    \multicolumn{1}{|c|}{}&\multicolumn{1}{|c|}{}& $\mathbf{lsv}$ & 1267.8~/~1278.0&1665.4~/~1678.0&2154.9~/~2165.0&2352.5~/~2485.0\\
    \multicolumn{1}{|c|}{}&\multicolumn{1}{|c|}{}& $\mathbf{cpu}$ & 16.6~/~17.1&26.1~/~26.9&19.8~/~20.0&24.9~/~27.0\\ \hline
    \end{tabular}
    \label{tab:self_time}
    }
    \end{adjustwidth}
\end{table}

Tables~\ref{tab:self_time} and \ref{tab:self_quality} show the results of our experiments to determine how PSPG performs as the parameters of the SPCP problem change. In Table~\ref{tab:self_time}, the average and maximum values for the statistics $\mathbf{cpu}$, $\mathbf{svd}$ and $\mathbf{lsv}$, which are taken over the $10$ random instances, are given for each choice of $n$, $c_r$ and $c_p$ values. For a given instance, $\mathbf{cpu}$ denotes the running time of \proc{PSPG} in
\emph{seconds}, $\mathbf{svd}$ denotes the number of partial SVDs computed during the run time and $\mathbf{lsv}$ denotes the total number of leading singular values and corresponding singular vectors computed, i.e., total number of singular values computed in all the partial SVDs during the run time. 
Table~\ref{tab:self_time} shows that the number of partial SVDs was almost constant regardless of the problem dimension $n$ and the problem parameters related to the rank and sparsity of $D$, i.e., $c_r$ and $c_p$. Moreover, Table~\ref{tab:self_quality} shows that the relative error of the solution $(L^{sol},S^{sol})$ was also almost constant for different $n$, $c_r$ and $c_p$ values.

\begin{table}[!htb]
    \begin{adjustwidth}{-2em}{-2em}
    \centering
    \caption{PSPG: Solution accuracy for decomposing $D\in\reals^{n\times n}$, $n\in\{500, 1000, 1500\}$}
    \renewcommand{\arraystretch}{1.5}
    {\scriptsize
    \begin{tabular}{ccc|c|c|c|c|}
    \cline{4-7}
    &&&$\mathbf{c_r}$=\textbf{0.05} $\mathbf{c_p}$=\textbf{0.05}&$\mathbf{c_r}$=\textbf{0.05} $\mathbf{c_p}$=\textbf{0.1}&$\mathbf{c_r}$=\textbf{0.1} $\mathbf{c_p}$=\textbf{0.05}&$\mathbf{c_r}$=\textbf{0.1} $\mathbf{c_p}$=\textbf{0.1}\\\hline
    \multicolumn{1}{|c|}{SNR}&\multicolumn{1}{|c|}{n}&Relative Error& \textbf{avg}~/~max & \textbf{avg}~/~max & \textbf{avg}~/~max & \textbf{avg}~/~max\\ \hline
    \multicolumn{1}{|c|}{\multirow{6}{*}{$80\mathrm{dB}$}}
    &\multicolumn{1}{|c|}{\multirow{2}{*}{$\mathbf{500}$}}
    & \textbf{relL} & 8.0E-05~/~8.0E-05&8.5E-05~/~8.5E-05&9.8E-05~/~1.0E-04&9.7E-05~/~9.9E-05 \\
    \multicolumn{1}{|c|}{}&\multicolumn{1}{|c|}{}& \textbf{relS} & 2.6E-04~/~2.7E-04&2.3E-04~/~2.3E-04&2.9E-04~/~2.9E-04&2.6E-04~/~2.7E-04\\ \cline{2-7}
    \multicolumn{1}{|c|}{}
    &\multicolumn{1}{|c|}{\multirow{2}{*}{$\mathbf{1000}$}}
    & \textbf{relL} & 9.9E-05~/~1.0E-04&9.7E-05~/~9.8E-05&1.1E-04~/~1.2E-04&1.1E-04~/~1.1E-04 \\
    \multicolumn{1}{|c|}{}&\multicolumn{1}{|c|}{}& \textbf{relS} & 2.7E-04~/~2.8E-04&2.3E-04~/~2.4E-04&3.5E-04~/~3.6E-04&2.8E-04~/~2.9E-04 \\ \cline{2-7}
    \multicolumn{1}{|c|}{}
    &\multicolumn{1}{|c|}{\multirow{2}{*}{$\mathbf{1500}$}}
    & \textbf{relL} & 1.0E-04~/~1.0E-04&9.7E-05~/~9.8E-05&1.1E-04~/~1.1E-04&1.3E-04~/~1.3E-04 \\
    \multicolumn{1}{|c|}{}&\multicolumn{1}{|c|}{}& \textbf{relS} & 2.7E-04~/~2.8E-04&2.3E-04~/~2.4E-04&3.5E-04~/~3.5E-04&3.5E-04~/~3.6E-04 \\ \hline
    \multicolumn{1}{|c|}{\multirow{6}{*}{$45\mathrm{dB}$}}
    &\multicolumn{1}{|c|}{\multirow{2}{*}{$\mathbf{500}$}}
    & \textbf{relL} & 4.5E-03~/~4.6E-03&4.8E-03~/~4.9E-03&5.6E-03~/~5.7E-03&5.5E-03~/~5.6E-03 \\
    \multicolumn{1}{|c|}{}&\multicolumn{1}{|c|}{}& \textbf{relS} & 1.5E-02~/~1.5E-02&1.3E-02~/~1.3E-02&1.7E-02~/~1.7E-02&1.5E-02~/~1.5E-02 \\ \cline{2-7}
    \multicolumn{1}{|c|}{}
    &\multicolumn{1}{|c|}{\multirow{2}{*}{$\mathbf{1000}$}}
    & \textbf{relL} & 5.2E-03~/~5.8E-03&5.5E-03~/~5.6E-03&6.5E-03~/~6.6E-03&6.3E-03~/~6.4E-03 \\
    \multicolumn{1}{|c|}{}&\multicolumn{1}{|c|}{}& \textbf{relS} & 1.4E-02~/~1.5E-02&1.3E-02~/~1.3E-02&2.0E-02~/~2.1E-02&1.6E-02~/~1.6E-02 \\ \cline{2-7}
    \multicolumn{1}{|c|}{}
    &\multicolumn{1}{|c|}{\multirow{2}{*}{$\mathbf{1500}$}}
    & \textbf{relL} & 5.9E-03~/~5.9E-03&5.5E-03~/~5.6E-03&6.5E-03~/~6.5E-03&6.8E-03~/~7.5E-03 \\
    \multicolumn{1}{|c|}{}&\multicolumn{1}{|c|}{}& \textbf{relS} & 1.6E-02~/~1.6E-02&1.3E-02~/~1.3E-02&2.0E-02~/~2.0E-02&1.8E-02~/~2.0E-02 \\ \hline
    \end{tabular}
    \label{tab:self_quality}
    }
    \end{adjustwidth}
\end{table}

Next, we compared PSPG with ASALM~\cite{Tao09_1J} for a fixed problem size, i.e. $n=1500$ where $D\in\reals^{n\times n}$. The code for \proc{ASALM} was obtained from the authors of~\cite{Tao09_1J}. The comparison results are displayed in Table~\ref{tab:compare_time} and Table~\ref{tab:compare_quality}. The statistics displayed in Table~\ref{tab:compare_time} and Table~\ref{tab:compare_quality}, i.e. \textbf{cpu}, \textbf{svd}, \textbf{lsv}, \textbf{relL} and \textbf{relS}, are defined at the beginning of this section.
\newpage
 Table~\ref{tab:compare_time} shows that for all of the problem classes, the number of partial SVDs required by \proc{PSPG} was slightly better than the number required by \proc{ASALM}. On the other hand, there was a big difference in CPU times; this difference can be explained by the fact that \proc{ASALM} required more leading singular value computations than \proc{PSPG} did per partial SVD. Table~\ref{tab:compare_quality} shows that the relative errors of the low-rank and sparse components produced by PSPG and ASALM are almost the same.

\begin{table}[!htb]
    \begin{adjustwidth}{-2em}{-2em}
    \centering
    \caption{PSPG vs ASALM: Solution time for decomposing $D\in\reals^{n\times n}$, $n=1500$}
    \renewcommand{\arraystretch}{1.3}
    {\scriptsize 
    \begin{tabular}{ccc|c|c|c|c|}
    \cline{4-7}
    &&&$\mathbf{c_r}$=\textbf{0.05} $\mathbf{c_p}$=\textbf{0.05}&$\mathbf{c_r}$=\textbf{0.05} $\mathbf{c_p}$=\textbf{0.1}&$\mathbf{c_r}$=\textbf{0.1} $\mathbf{c_p}$=\textbf{0.05}&$\mathbf{c_r}$=\textbf{0.1} $\mathbf{c_p}$=\textbf{0.1}\\\hline
    \multicolumn{1}{|c|}{SNR}&\multicolumn{1}{|c|}{Alg.}& Field & \textbf{avg}~/~max & \textbf{avg}~/~max & \textbf{avg}~/~max & \textbf{avg}~/~max\\ \hline
    \multicolumn{1}{|c|}{\multirow{6}{*}{$80\mathrm{dB}$}}
    &\multicolumn{1}{|c|}{\multirow{3}{*}{$\mathbf{PSPG}$}}
    & $\mathbf{svd}$ & 26~/~26&27~/~27&25~/~25&25~/~25 \\
    \multicolumn{1}{|c|}{}&\multicolumn{1}{|c|}{}& $\mathbf{lsv}$ & 2084.8~/~2094.0&2474.4~/~2492.0&3685.1~/~3695.0&3690.3~/~3703.0\\
    \multicolumn{1}{|c|}{}&\multicolumn{1}{|c|}{}& $\mathbf{cpu}$ & 28.7~/~32.8&41.3~/~45.2&34.0~/~34.3&38.2~/~38.6\\ \cline{2-7}
    \multicolumn{1}{|c|}{}
    &\multicolumn{1}{|c|}{\multirow{3}{*}{$\mathbf{ASALM}$}}
    & $\mathbf{svd}$ & 29.2~/~30&32.5~/~33&29.8~/~30&35.3~/~36 \\
    \multicolumn{1}{|c|}{}&\multicolumn{1}{|c|}{}& $\mathbf{lsv}$ & 3577.0~/~3643.0&5763.9~/~5809.0&6017.8~/~6059.0&9544.3~/~9651.0\\
    \multicolumn{1}{|c|}{}&\multicolumn{1}{|c|}{}& $\mathbf{cpu}$ & 45.8~/~47.5&77.9~/~80.1&63.1~/~64.4&106.8~/~108.4\\ \hline
    \multicolumn{1}{|c|}{\multirow{6}{*}{$45\mathrm{dB}$}}
    &\multicolumn{1}{|c|}{\multirow{3}{*}{$\mathbf{PSPG}$}}
    & $\mathbf{svd}$ & 16~/~16&17~/~17&15~/~15&15.6~/~16 \\
    \multicolumn{1}{|c|}{}&\multicolumn{1}{|c|}{}& $\mathbf{lsv}$ & 1267.8~/~1278.0&1665.4~/~1678.0&2154.9~/~2165.0&2352.5~/~2485.0\\
    \multicolumn{1}{|c|}{}&\multicolumn{1}{|c|}{}& $\mathbf{cpu}$ & 16.6~/~17.1&26.1~/~26.9&19.8~/~20.0&24.9~/~27.0\\ \cline{2-7}
    \multicolumn{1}{|c|}{}
    &\multicolumn{1}{|c|}{\multirow{3}{*}{$\mathbf{ASALM}$}}
    & $\mathbf{svd}$ & 13.6~/~14&20.9~/~21&15~/~15&22.6~/~23 \\
    \multicolumn{1}{|c|}{}&\multicolumn{1}{|c|}{}& $\mathbf{lsv}$ & 2522.0~/~2597.0&5555.8~/~5607.0&3839.2~/~3855.0&7857.2~/~7979.0\\
    \multicolumn{1}{|c|}{}&\multicolumn{1}{|c|}{}& $\mathbf{cpu}$ & 33.0~/~34.5&76.2~/~79.6&44.1~/~44.5&93.5~/~96.0\\ \hline
    \end{tabular}
    \label{tab:compare_time}
    }
    \end{adjustwidth}
\end{table}
\begin{table}[!htb]
    \begin{adjustwidth}{-2em}{-2em}
    \centering
    \caption{PSPG vs ASALM: Solution accuracy for decomposing $D\in\reals^{n\times n}$, $n=1500$}
    \renewcommand{\arraystretch}{1.5}
    {\scriptsize 
    \begin{tabular}{ccc|c|c|c|c|}
    \cline{4-7}
    &&&$\mathbf{c_r}$=\textbf{0.05} $\mathbf{c_p}$=\textbf{0.05}&$\mathbf{c_r}$=\textbf{0.05} $\mathbf{c_p}$=\textbf{0.1}&$\mathbf{c_r}$=\textbf{0.1} $\mathbf{c_p}$=\textbf{0.05}&$\mathbf{c_r}$=\textbf{0.1} $\mathbf{c_p}$=\textbf{0.1}\\\hline
    \multicolumn{1}{|c|}{SNR}&\multicolumn{1}{|c|}{Alg.}&Relative Error& \textbf{avg}~/~max & \textbf{avg}~/~max & \textbf{avg}~/~max & \textbf{avg}~/~max\\ \hline
    \multicolumn{1}{|c|}{\multirow{4}{*}{$80\mathrm{dB}$}}
    &\multicolumn{1}{|c|}{\multirow{2}{*}{$\mathbf{PSPG}$}}
    & \textbf{relL}
    &1.0E-04~/~1.0E-04&9.7E-05~/~9.8E-05&1.1E-04~/~1.1E-04&1.3E-04~/~1.3E-04\\
    \multicolumn{1}{|c|}{}&\multicolumn{1}{|c|}{}& \textbf{relS}
    &2.7E-04~/~2.8E-04&2.3E-04~/~2.4E-04&3.5E-04~/~3.5E-04&3.5E-04~/~3.6E-04\\ \cline{2-7}
    \multicolumn{1}{|c|}{}
    &\multicolumn{1}{|c|}{\multirow{2}{*}{$\mathbf{ASALM}$}}
    & \textbf{relL}
    &4.5E-05~/~4.7E-05&4.7E-05~/~4.9E-05&7.6E-05~/~7.8E-05&7.8E-05~/~8.0E-05\\
    \multicolumn{1}{|c|}{}&\multicolumn{1}{|c|}{}& \textbf{relS}
    &4.6E-04~/~5.0E-04&3.3E-04~/~3.5E-04&6.4E-04~/~6.7E-04&4.4E-04~/~4.6E-04\\ \hline
    \multicolumn{1}{|c|}{\multirow{4}{*}{$45\mathrm{dB}$}}
    &\multicolumn{1}{|c|}{\multirow{2}{*}{$\mathbf{PSPG}$}}
    & \textbf{relL}
    &5.9E-03~/~5.9E-03&5.5E-03~/~5.6E-03&6.5E-03~/~6.5E-03&6.8E-03~/~7.5E-03\\
    \multicolumn{1}{|c|}{}&\multicolumn{1}{|c|}{}& \textbf{relS}
    &1.6E-02~/~1.6E-02&1.3E-02~/~1.3E-02&2.0E-02~/~2.0E-02&1.8E-02~/~2.0E-02\\ \cline{2-7}
    \multicolumn{1}{|c|}{}
    &\multicolumn{1}{|c|}{\multirow{2}{*}{$\mathbf{ASALM}$}}
    & \textbf{relL}
    &2.5E-03~/~2.9E-03&2.9E-03~/~3.0E-03&4.2E-03~/~4.3E-03&4.3E-03~/~5.2E-03\\
    \multicolumn{1}{|c|}{}&\multicolumn{1}{|c|}{}& \textbf{relS}
    &2.0E-02~/~2.8E-02&1.3E-02~/~1.4E-02&3.0E-02~/~3.1E-02&2.0E-02~/~2.6E-02\\ \hline
    \end{tabular}
    \label{tab:compare_quality}
    }
    \end{adjustwidth}
\end{table}

\subsection{Foreground Detection on a Noisy Video}
\label{sec:video_test_results}
As described in section~\ref{sec:video_test_results-RPCA}, foreground extraction from a noisy video can be formulated as SPCP problem. We used \proc{PSPG} and \proc{ASALM} to extract moving objects in a surveillance  video~\cite{Li04_1J}, which is a sequence of $201$
grayscale $144 \times 176$  frames. We represented this information as
a data matrix $D \in \reals^{(144\times 176)\times 201}$, where the
$i$-th column of $D$ was constructed by stacking the
columns of the $i$-th frame into a long vector.
We assumed that the original airport
security video was noiseless and created a noisy video sequence with
$\proc{SNR}=20$dB as follows. We set $\varrho = \norm{D}_F/(\sqrt{144\times 176\times
  201}~10^{\proc{SNR}/20})$ and then added to each component $D_{ij}$
of the data matrix an independent sample from a Normal distribution
with mean zero and variance $\varrho^2$.

We solved the corresponding  SPCP problem using PSPG and ASALM.
We set  $i$-th background  frame to be  $i$-th
column of the recovered low-rank matrix $L$, and the $i$-th foreground
frame to be the $i$-th column of the sparse matrix $S$. PSPG and ASALM were terminated when
$\frac{\norm{(L_{k+1},S_{k+1})-(L_{k},S_{k})}_F}{(\norm{(L_{k},S_{k})}_F+1)\varrho}$ is less than
$5\times10^{-4}$ and $1\times10^{-4}$, respectively, in order to obtain similar visual quality in reconstruction.
\begin{figure} [t!]
    \centering
    \mbox{\hspace{4mm}$D(t)$:}
    \includegraphics[scale=0.6]{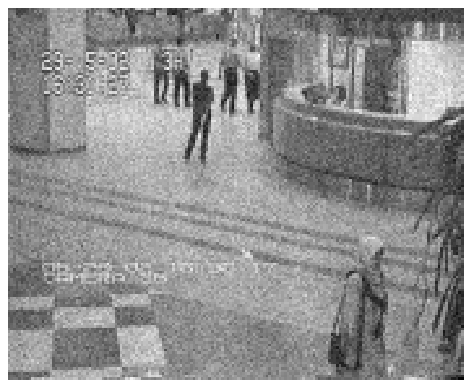}
    \includegraphics[scale=0.6]{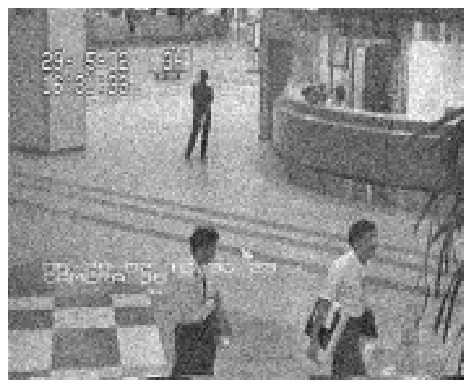}
    \includegraphics[scale=0.6]{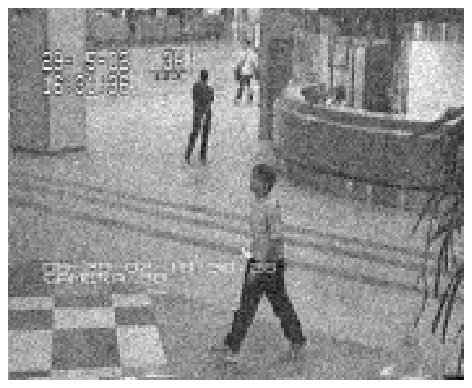}\\
    \mbox{$L^{sol}(t)$:}
    \includegraphics[scale=0.6]{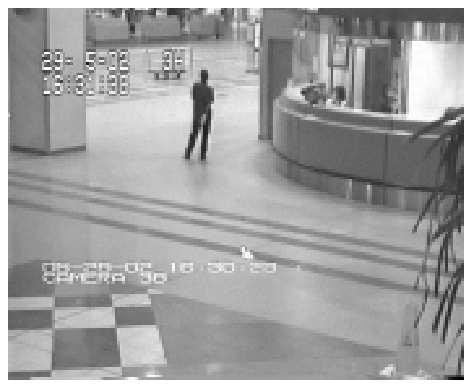}
    \includegraphics[scale=0.6]{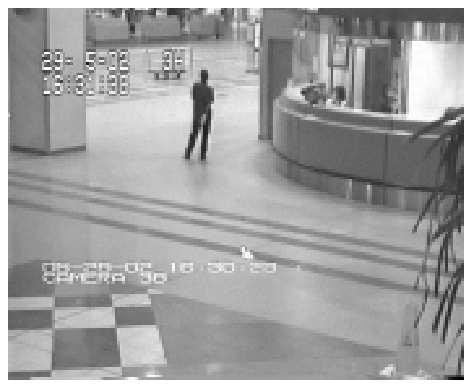}
    \includegraphics[scale=0.6]{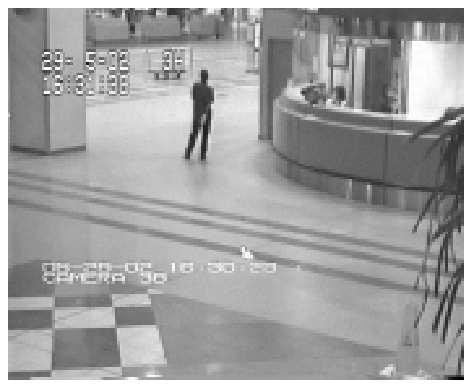}\\
    \mbox{$S^{sol}(t)$: }
    \includegraphics[scale=0.6]{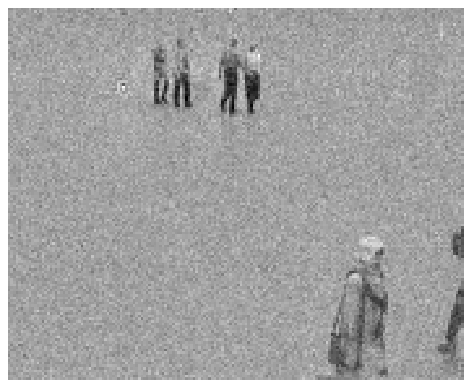}
    \includegraphics[scale=0.6]{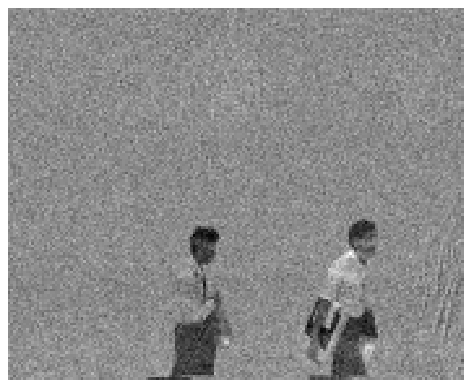}
    \includegraphics[scale=0.6]{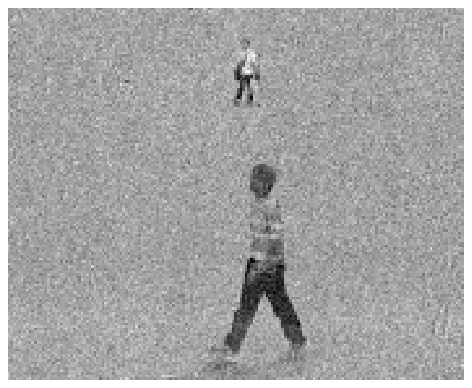}\\
    \mbox{$S_{post}^{sol}(t)$: }
    \includegraphics[scale=0.6]{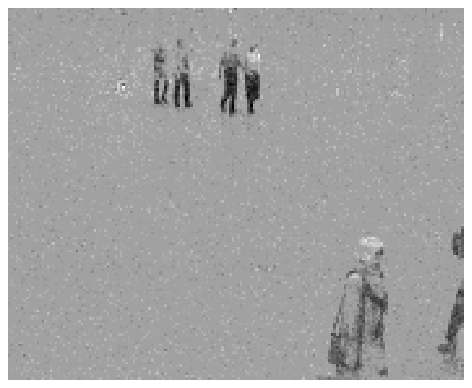}
    \includegraphics[scale=0.6]{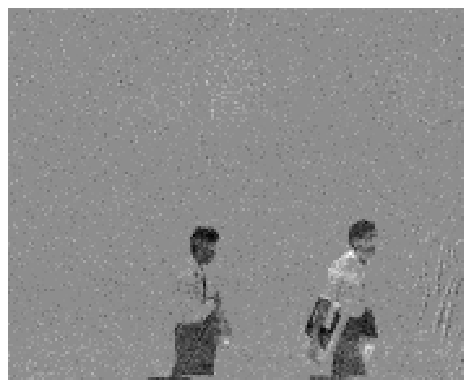}
    \includegraphics[scale=0.6]{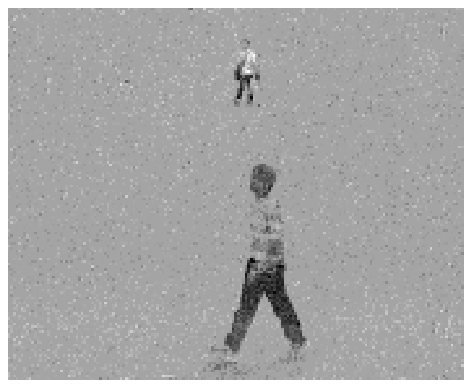}
    \caption{Background extraction from a video with $20$dB SNR using \proc{PSPG}}
    \label{fig:noisy_reconstruction_test_pspg}
\end{figure}
The recovery statistics for each algorithm are displayed in
Table~\ref{tab:compare_video}. $(L^{sol},S^{sol})$ denotes the
variables corresponding to the low-rank and sparse components of $D$,
respectively, when the algorithm of interest
terminates. The first row in
Figure~\ref{fig:noisy_reconstruction_test_pspg} displays the
$35$-th, $100$-th and $125$-th frames of the noisy surveillance
video~\cite{Li04_1J}. The second and third
rows display the recovered background and foreground images of the
selected frames, respectively, using PSPG. The frames recovered by
ASALM were very similar to those of PSPG.  Even though the visual
quality of recovered background and foreground are very similar for
both algorithms, Table~\ref{tab:compare_video} shows that both the
number of partial SVDs and the CPU time of PSPG are significantly smaller
than those of \proc{ASALM}.

\begin{table}[!htb]
    \begin{adjustwidth}{-2em}{-2em}
    \centering
    \caption{PSPG vs ASALM: Recovery statistics for foreground detection on a noisy video}
    \renewcommand{\arraystretch}{1.75}
    {\scriptsize 
    \begin{tabular}{|c|c|c|c|c|c|c|}
    \hline
    Alg.& \textbf{svd} & \textbf{lsv} & \textbf{cpu}& $\mathbf{\norm{L^{sol}}_*}$ & $\mathbf{\norm{S^{sol}}_1}$ & $\mathbf{\rank(L^{sol})}$ \\ \hline
    $\mathbf{PSPG}$ & 18 & 296 &54.4 & $\approx3.5\times 10^5$ & $\approx 7.1\times 10^7$ & 1 \\ \hline
    $\mathbf{ASALM}$ &  56 & 2232 & 93.1 & $\approx4.0\times 10^5$ & $\approx7.6\times 10^7$ & 84 \\ \hline
    \end{tabular}
    \label{tab:compare_video}
    }
    \end{adjustwidth}
\end{table}

In our preliminary numerical experiments, we noticed that the
recovered background frames are almost noise free even when the input
video was very noisy, and almost all the
noise shows up in the recovered foreground images. This was
observed for both PSPG and ASALM. Hence, in order to
eliminate the noise seen in the recovered foreground frames and
enhance the quality of the recovered frames, we post-process
$(L^{sol},S^{sol})$ of PSPG as follows:
\begin{align}
\label{eq:postprocess}
S_{post}^{sol}:=\argmin_S\{\norm{S}_1:~\norm{S+L^{sol}-D}_F\leq\delta\}.
\end{align}
The fourth row of Figure~\ref{fig:noisy_reconstruction_test_pspg} shows
the post-processed foreground frames and $\norm{S_{post}^{sol}}_1\approx 3.2\times 10^7$.
\section{Conclusion}
In this paper, we proposed proximal gradient and alternating linearization methods for solving robust and stable PCA problems. We proved that $O(1/\epsilon)$ iterations are required to obtain an $\epsilon$-optimal solution to the nonsmooth RPCP and SPCP problems. Numerical results on problems with arising from background separation from surveillance video, shadow and specularity removal from face images and video denoising from heavily corrupted data are reported. The results show that our methods are able to solve huge problems involving million variables and linear constraints effectively.
\bibliographystyle{siam}
\bibliography{All}
\appendix

\section{Proof of Lemma~\ref{eq:subproblem_nsa}}
\label{sec:appendix2}
\begin{proof}
Suppose that $\delta>0$. Let $(L^*,S^*)$ be an optimal solution to problem $(P_{ns})$ 
and $\theta^*$ denote the optimal Lagrangian multiplier for the constraint $(L,S)\in\chi$ written as

\begin{align}
\frac{1}{2}\norm{L+S-\pi_\Omega(D)}^2_F\leq \frac{\delta^2}{2}. \label{eq:quadratic_constraint}
\end{align}
Then the KKT optimality conditions for this problem are
\begin{enumerate}[i.]
    \item $Q+\rho^{-1}(L^*-\tilde{L})+\theta^*(L^*+S^*-\pi_\Omega(D))=0$, \label{condition1}
    \item $\xi G + \theta^*(L^*+S^*-\pi_\Omega(D))=0$ and $G\in\pi_\Omega\left(\partial\norm{\pi_\Omega(S^*)}_1\right)$, \label{condition2}
    \item $\norm{L^*+S^*-\pi_\Omega(D)}_F\leq\delta$, \label{condition3}
    \item $\theta^*\geq 0$, \label{condition4}
    \item $\theta^*~(\norm{L^*+S^*-\pi_\Omega(D)}_F-\delta)=0$, \label{condition5}
\end{enumerate}
where \ref{condition2} follows from $\pi^*_\Omega=\pi_\Omega$. From \ref{condition1} and \ref{condition2}, we have
\begin{eqnarray}
\left[
  \begin{array}{cc}
    (\rho^{-1}+\theta^*)I &  \theta^*I\\
    \theta^*I & \theta^*I \\
  \end{array}
\right]
\left[
  \begin{array}{c}
    L^* \\
    S^* \\
  \end{array}
\right]
=
\left[
  \begin{array}{c}
    \theta^*\pi_\Omega(D)+\rho^{-1}~q(\tilde{L}) \\
    \theta^*\pi_\Omega(D)-\xi G\\
  \end{array}
\right], \label{eq:FTOC_1}
\end{eqnarray}
where $q(\tilde{L})=\tilde{L}-\rho~Q$. From \eqref{eq:FTOC_1} it follows that
\begin{eqnarray}
\left[
  \begin{array}{cc}
    (\rho^{-1}+\theta^*)I &  \theta^*I\\
    0 & \left(\frac{\theta^*}{1+\rho\theta^*}\right)~I \\
  \end{array}
\right]
\left[
  \begin{array}{c}
    L^* \\
    S^* \\
  \end{array}
\right]
=
\left[
  \begin{array}{c}
    \theta^*\pi_\Omega(D)+\rho^{-1}~q(\tilde{L}) \\
    \frac{\theta^*}{1+\rho\theta^*}~\left(\pi_\Omega(D)-q(\tilde{L})\right)-\xi G\\
  \end{array}
\right]. \label{eq:FTOC_2}
\end{eqnarray}
From the second equation in \eqref{eq:FTOC_2}, we have
\begin{align}
\xi\frac{(1+\rho\theta^*)}{\theta^*}~G+S^*+q(\tilde{L})-\pi_\Omega(D)=0. \label{eq:shrinkS}
\end{align}
But \eqref{eq:shrinkS} is precisely the first-order optimality conditions for the ``shrinkage" problem $$\min_{S\in\reals^{m\times  n}}\left\{\xi\frac{(1+\rho\theta^*)}{\theta^*}\norm{\pi_\Omega(S)}_1+\frac{1}{2}\norm{S+q(\tilde{L})-\pi_\Omega(D)}_F^2\right\}.$$ Thus, $S^*$ is the optimal solution to the ``shrinkage" problem and is given by \eqref{lemeq:S}.
\eqref{lemeq:L} follows from the first equation in \eqref{eq:FTOC_2}, and it implies
\begin{align}
L^*+S^*-\pi_\Omega(D) = \frac{1}{1+\rho\theta^*}~\pi_\Omega(S^*+q(\tilde{L})-D).
\end{align}
Therefore,
\begin{align}
\norm{L^*+S^*-\pi_\Omega(D)}_F &= \frac{1}{1+\rho\theta^*}~\norm{\pi_\Omega\left(S^*+q(\tilde{L})-D\right)}_F, \nonumber\\
&=\frac{1}{1+\rho\theta^*}~\norm{\pi_\Omega\left(\sgn\left(D-q(\tilde{L})\right)\odot\max\left\{|D-q(\tilde{L})|-\xi\frac{(1+\rho\theta^*)}{\theta^*}~E,\ \mathbf{0}\right\}-\left(D-q(\tilde{L})\right)\right)}_F, \nonumber\\
&=\frac{1}{1+\rho\theta^*}~\norm{\pi_\Omega\left(\max\left\{|D-q(\tilde{L})|-\xi\frac{(1+\rho\theta^*)}{\theta^*}~E,\ \mathbf{0}\right\}-|D-q(\tilde{L})|\right)~}_F,\nonumber\\
&=\frac{1}{1+\rho\theta^*}~\norm{\pi_\Omega\left(\min\left\{\xi\frac{(1+\rho\theta^*)}{\theta^*}~E,\ |D-q(\tilde{L})|\right\}\right)}_F,\nonumber\\
&=\norm{\min\left\{\frac{\xi}{\theta^*}~E,\ \frac{1}{1+\rho\theta^*}~\left|\pi_\Omega\left(D-q(\tilde{L})\right)\right|\right\}}_F, \label{eq:Fnorm}
\end{align}
where the second equation uses \eqref{lemeq:S}. Now let $\phi:\reals_+\rightarrow\reals_+$ be
\begin{align}
\phi(\theta):= \norm{\min\left\{\frac{\xi}{\theta}~E,\ \frac{1}{1+\rho\theta}~\left|\pi_\Omega\left(D-q(\tilde{L})\right)\right|\right\}}_F.
\end{align}

\subsection*{Case 1: $\norm{\pi_\Omega\left(D-q(\tilde{L})\right)}_F\leq\delta$}
$\theta^*=0$, $L^*=q(\tilde{L})$ and $S^*=-\pi_{\Omega^c}\left(q(\tilde{L})\right)$ trivially satisfy all the KKT conditions.
\subsection*{Case 2: $\norm{\pi_\Omega\left(D-q(\tilde{L})\right)}_F>\delta$}
It is easy to show that $\phi(.)$ is a strictly decreasing function of $\theta$. Since $\phi(0)=\norm{\pi_\Omega\left(D-q(\tilde{L})\right)}_F>\delta$ and $\lim_{\theta\rightarrow\infty}\phi(\theta)=0$, there exists a unique $\theta^*>0$ such that $\phi(\theta^*)=\delta$. Given $\theta^*$, $S^*$ and $L^*$ can then be computed from equations \eqref{lemeq:S} and \eqref{lemeq:L}, respectively. Moreover, since $\theta^*>0$ and $\phi(\theta^*)=\delta$, \eqref{eq:Fnorm} implies that $L^*$, $S^*$ and $\theta^*$ satisfy the KKT conditions.

We now show that $\theta^*$ can be computed in $\cO(|\Omega|\log(|\Omega|))$ time. Let $A:=\pi_\Omega\left(|D-q(\tilde{L})|\right)$ and $0\leq a_{(1)}\leq a_{(2)}\leq ... \leq a_{(|\Omega|)}$ be the $|\Omega|$ elements of the matrix $A$ corresponding to the indices $(i,j)\in\Omega$ sorted in increasing order, which can be done in $\cO(|\Omega|\log(|\Omega|))$ time. Defining $a_{(0)}:=0$ and $a_{(|\Omega|+1)}:=\infty$, we then have for all $j\in\{0,1,...,|\Omega|\}$ that
\begin{align}
\frac{1}{1+\rho\theta}~a_{(j)} \leq \frac{\xi}{\theta} \leq \frac{1}{1+\rho\theta}~a_{(j+1)} \Leftrightarrow \frac{1}{\xi}~a_{(j)}-\rho \leq \frac{1}{\theta} \leq \frac{1}{\xi}~a_{(j+1)}-\rho.
\end{align}
For all $\bar{k}< j\leq |\Omega|$ define $\theta_j$ such that $\frac{1}{\theta_j}=\frac{1}{\xi}~a_{(j)}-\rho$ and let $\bar{k}:=\max\left\{j: \frac{1}{\theta_j}\leq 0,\ j\in\{0,1,...,|\Omega|\}\right\}$. Then for all $\bar{k}< j\leq |\Omega|$
\begin{align}
\phi(\theta_j)=\sqrt{\left(\frac{1}{1+\rho\theta_j}\right)^2~\sum_{i=0}^j a^2_{(i)}+(|\Omega|-j)~\left(\frac{\xi}{\theta_j}\right)^2}.
\end{align}
Also define $\theta_{\bar{k}}:=\infty$ and $\theta_{|\Omega|+1}:=0$ so that $\phi(\theta_{\bar{k}}):=0$ and $\phi(\theta_{|\Omega|+1})=\phi(0)=\norm{A}_F>\delta$. Note that $\{\theta_j\}_{\{\bar{k}< j\leq |\Omega|\}}$ contains all the points at which $\phi(\theta)$ may not be differentiable for $\theta\geq 0$.
Define $j^*:=\max\{j:\ \phi(\theta_j)\leq\delta,\ \bar{k}\leq j\leq |\Omega|\}$. Then $\theta^*$ is the unique solution of the system
\begin{align}
\label{eq:root}
\sqrt{\left(\frac{1}{1+\rho\theta}\right)^2~\sum_{i=0}^{j^*} a^2_{(i)}+(|\Omega|-j^*)~\left(\frac{\xi}{\theta}\right)^2}=\delta \mbox{ and } \theta>0,
\end{align}
since $\phi(\theta)$ is continuous and strictly decreasing in $\theta$ for $\theta\geq 0$. Solving the equation in \eqref{eq:root} requires finding the roots of a fourth-order polynomial (a.k.a. quartic function); therefore, one can compute $\theta^*>0$ using the algebraic solutions of quartic equations (as shown by Lodovico Ferrari in 1540), which requires $\cO(1)$ operations. Note that if $\bar{k}=|\Omega|$, then $\theta^*$ is the solution of the equation
\begin{align}
\sqrt{\left(\frac{1}{1+\rho\theta^*}\right)^2~\sum_{i=1}^{|\Omega|} a^2_{(i)}}=\delta,
\end{align}
i.e. $\theta^*= \rho^{-1}\left(\frac{\norm{A}_F}{\delta}-1\right)=\rho^{-1}\left(\frac{\norm{\pi_\Omega\left(D-\tilde{L}\right)}_F}{\delta}-1\right)$.
Hence, we have proved that problem~$(P_{ns})$ can be solved efficiently.

Now, suppose that $\delta=0$. Since $L^*+S^*=\pi_\Omega(D)$, problem~$(P_{ns})$ 
can be written as
\begin{equation}
\label{eq:subproblem_delta0}
\begin{array}{ll}
\min_{S\in\reals^{m\times n}} & \xi\rho\norm{\pi_\Omega(S)}_1+\frac{1}{2}\norm{S+q(\tilde{L})-\pi_\Omega(D)}_F^2.
\end{array}
\end{equation}
Then \eqref{lemeq:XS_nonsmooth_delta0} trivially follows from first-order optimality conditions for the above problem and the fact that $L^*=\pi_\Omega(D)-S^*$.
\end{proof}

\end{document}